\documentclass[12pt,letterpaper,titlepage,reqno]{amsart}
\usepackage{amsmath, amssymb, amsthm, amsfonts,amscd,amsaddr,enumerate,mathtools}

\usepackage{hyperref}

\usepackage{backref}

\usepackage[
    paper=a4paper,
    portrait=true,
    textwidth=425pt,
    textheight=650pt,
         tmargin=3cm,
    marginratio=1:1
            ]{geometry}

\theoremstyle{plain}
\newtheorem{theorem}{Theorem}[section]
\newtheorem{proposition}[theorem]{Proposition}
\newtheorem{cor}[theorem]{Corollary}

\newtheorem{prop}[theorem]{Proposition}
\newtheorem{lemma}[theorem]{Lemma}

\theoremstyle{definition}

\newtheorem{rmk}[theorem]{Remark}
\numberwithin{equation}{section}
\newtheorem*{theoremA*}{Theorem A}
\newtheorem*{theoremB*}{Theorem B}
\newtheorem*{theorem1*}{Theorem A'}
\newtheorem*{theoremC*}{Theorem C}
\newtheorem*{theoremD*}{Theorem D}
\newtheorem*{theoremE*}{Theorem E}
\newtheorem*{theoremF*}{Theorem F}
\newtheorem*{theoremE2*}{Theorem E2}
\newtheorem*{theoremE3*}{Theorem E3}
\newcommand{\bs}{\backslash}

\newcommand{\C}{\mathbb{C}}

\newcommand{\Lc}{\mathcal{L}}
\newcommand{\E}{\mathcal{E}}
\newcommand{\Hc}{\mathcal{H}}
\newcommand{\Hb}{\mathbb{H}}
\newcommand{\Q}{\mathbb{Q}}

\newcommand{\Oc}{\mathcal{O}}

\newcommand{\Pc}{\mathcal{P}}
\newcommand{\R}{\mathbb{R}}
\newcommand{\N}{\mathbb{N}}

\newcommand{\SO}{\operatorname{SO}}

\newcommand{\GL}{\operatorname{GL}}
\newcommand{\SL}{\operatorname{SL}}

\newcommand{\SU}{\operatorname{SU}}

\newcommand{\im}{\operatorname{Im}}
\newcommand{\Sp}{\operatorname{Sp}}

\newcommand{\Ad}{\operatorname{Ad}}

\newcommand{\ad}{\operatorname{ad}}
\newcommand{\diag}{\operatorname{diag}}

\newcommand{\Pol}{\operatorname{Pol}}

\newcommand{\res} {\operatorname{Res}}
\newcommand{\Spec}{\operatorname{spec}}

\newcommand{\re}{\operatorname{Re}}

\newcommand{\Mat}{\operatorname{Mat}}

\def\hat{\widehat}
\def\af{\mathfrak{a}}

\def\gf{\mathfrak{g}}

\def\kf{\mathfrak{k}}
\def\lf{\mathfrak{l}}
\def\mf{\mathfrak{m}}
\def\nf{\mathfrak{n}}

\def\sf{\mathfrak{s}}

\def\gl{\mathfrak{gl}}

\def\uf{\mathfrak{u}}

\def\zf{\mathfrak{z}}
\def\la{\langle}
\def\ra{\rangle}
\def\1{{\bf1}}

\def\U{\mathcal{U}}

\def\B{\mathcal{B}}
\def\Cc{\mathcal{C}}
\def\Tc{\mathcal{T}}

\def\Oc{\mathcal{O}}

\def\oline{\overline}

\def\tilde{\widetilde}

\def\Pol{\operatorname{Pol}}
\def\tilde{\widetilde}

\def\oline{\overline}
\def\la{\langle}
\def\ra{\rangle}
\usepackage[usenames]{color}

\title[Poisson transform]
{Poisson transform and unipotent complex geometry}

\begin{document}

\begin{abstract} Our concern is with Riemannian symmetric spaces $Z=G/K$
of the non-compact type and more precisely with the Poisson transform $\Pc_\lambda$  which maps
generalized functions on the boundary $\partial Z$ to $\lambda$-eigenfunctions on $Z$. Special emphasis is given to a maximal unipotent group $N<G$ which naturally
acts on both $Z$ and $\partial Z$. The $N$-orbits on $Z$ are parametrized
by a torus $A=(\R_{>0})^r<G$ (Iwasawa) and letting the level $a\in A$ tend to $0$ on a ray we retrieve
$N$ via $\lim_{a\to 0} Na$ as an open dense orbit in $\partial Z$ (Bruhat).
For positive parameters $\lambda$ the Poisson transform $\Pc_\lambda$
is defined an{ d} injective for functions $f\in L^2(N)$ and we give a novel characterization of
$\Pc_\lambda(L^2(N))$ in terms of complex analysis. For that we view
eigenfunctions $\phi = \Pc_\lambda(f)$ as families $(\phi_a)_{a\in A}$ of functions
on the $N$-orbits, i.e. $\phi_a(n)= \phi(na)$ for $n\in N$. The general theory then tells us that
there is a tube domain $\Tc=N\exp(i\Lambda)\subset N_\C$ such that each $\phi_a$
extends to a holomorphic function on the scaled tube $\Tc_a=N\exp(i\Ad(a)\Lambda)$.
We define a  class of $N$-invariant weight functions { ${\bf w}_\lambda$ on the tube $\Tc$},
rescale them for every $a\in A$ to a weight ${\bf w}_{\lambda, a}$ on $\Tc_a$, and
show that each $\phi_a$ lies in the $L^2$-weighted Bergman space
$\B(\Tc_a, {\bf w}_{\lambda, a}):=\Oc(\Tc_a)\cap L^2(\Tc_a, {\bf w}_{\lambda, a})$. The main result of the article then describes $\Pc_\lambda(L^2(N))$ as those eigenfunctions $\phi$
for which $\phi_a\in \B(\Tc_a, {\bf w}_{\lambda, a})$ and
$$\|\phi\|:=\sup_{a\in A} a^{\re\lambda -2\rho} \|\phi_a\|_{\B_{a,\lambda}}<\infty$$
holds.
\end{abstract}

\author[Gimperlein]{Heiko Gimperlein}
\address{Engineering Mathematics\\
Leopold-Franzens-Universit\"at Innsbruck\\
6020 Innsbruck, Austria\\
{\tt heiko.gimperlein@uibk.ac.at}}
\author[Kr\"otz]{Bernhard Kr\"otz}
\address{Institut f\"ur Mathematik\\
Universit\"at Paderborn\\Warburger Str. 100,
33098 Paderborn, Germany \\ {\tt bkroetz@gmx.de}}

\author[Roncal]{Luz Roncal}
\address{BCAM - Basque Center for Applied Mathematics\\
 48009 Bilbao, Spain and\\ Ikerbasque
Basque Foundation for Science, 48011 Bilbao, Spain and\\
Universidad del Pa\'is Vasco / Euskal Herriko Unibertsitatea,
48080 Bilbao, Spain\\
{\tt lroncal@bcamath.org}}

\author[Thangavelu]{Sundaram Thangavelu}
\address{Department of Mathematics\\
 Indian Institute of Science\\
560 012 Bangalore, India\\
{\tt veluma@iisc.ac.in}}

\maketitle
\section{Introduction}

This article considers range theorems for the Poisson transform on Riemannian symmetric spaces $Z$ in the context of horospherical complex geometry.


We assume that $Z$ is of non-compact type and let $G$ be the semisimple Lie group of isometries of $Z$. Then $Z$ is homogeneous for $G$ and identified
as $Z=G/K$, where $K\subset G$ is a maximal compact subgroup and stabilizer of a fixed base point $z_0\in Z$. Classical examples are the real hyperbolic spaces which will
receive special explicit attention at the end of the article.
\par The Poisson transform maps sections of line bundles over the compact boundary
$\partial Z$ to eigenfunctions of the commutative algebra of $G$-invariant
differential operators $\mathbb{D}(Z)$ on $Z$. Recall that $\partial Z = G/{ \oline P}$ is a real flag manifold
for ${ \oline P =MA\oline N}$ a minimal parabolic subgroup originating from an Iwasawa decomposition
${ G=KA\oline N}$ of $G$. The line bundles we consider are parametrized by the complex characters $\lambda$ of the abelian group
$A$, and we write $\Pc_\lambda$ for the corresponding Poisson transform. { We let $N $ be the unipotent radical of the parabolic subgroup $P=MAN$ opposed to $\oline P$. }

\par The present paper initiates the study of the Poisson transform in terms of the $N$-geometry of both $Z$ and
$\partial Z$. Identifying the contractible group $N$ with its open dense orbit in $\partial Z$, functions on $N$ correspond to sections of the line bundle via extension by zero.  On the other hand
$N\bs Z\simeq A$. 
Hence, given a function $f\in L^2(N)$ with Poisson transform $\phi { =} \Pc_\lambda(f)$, it is natural to consider the family $\phi_a$, $a\in A \simeq N\bs Z$, of functions restricted to the $N$-orbits $Na\cdot z_0\subset Z$.
A basic observation then is that the functions $\phi_a$ extend holomorphically to $N$-invariant tubular neighborhoods $\Tc_a\subset N_\C$ of $N$. Our main result, Theorem \ref{maintheorem}, identifies for positive parameters $\lambda$  the image $\Pc_\lambda(L^2(N))$ with a class of families $\phi_a$ in
weighted Bergman spaces  $\B(\Tc_a, {\bf w}_{\lambda, a})$ on these tubes $\Tc_a$. 

\par Range theorems for the Poisson transform in terms of the $K$-geometry of both
$\partial Z$ and $Z$ were investigated in \cite{I} for spaces of rank one. Note that $\partial Z\simeq  K/M$ and that every line bundle over $K/M$ is trivial, so that sections can be identified with functions on $K/M$. On the other hand
$K\bs Z\simeq A/W$ with $W$ the little Weyl group, a finite reflection group. Given
a  function $f \in L^2(K/M)$ the image $\phi=\Pc_\lambda(f)$ therefore induces a family of partial functions $\phi_a: K\to \C$ with $\phi_a(k):=\phi(ka\cdot z_0)$ on the $K$-orbits in $Z$ parametrized by
$a\in A$.
As $\phi$ is continuous, we have $\phi_a\in L^2(K)$, and \cite{I} characterizes the image $\Pc_\lambda(L^2(K/M))$
in terms of the growth of $\|\phi_a\|_{L^2(K)}$ and suitable maximal functions. Interesting follow up work includes
\cite{BOS} and \cite{Ka}.

\bigskip To explain our results in more detail, we first describe our perspective on eigenfunctions of the algebra $\mathbb{D}(Z)$. The Iwasawa decomposition $G=KAN$ allows us to identify $Z=G/K$ with the solvable group
$S=NA$. Inside $\mathbb{D}(Z)$ one finds a distinguished element, the Laplace--Beltrami
operator $\Delta_Z$. Upon identifying $Z$ with $S$ we use the symbol
$\Delta_S$ instead of $\Delta_Z$. Now it is a remarkable fact that all $\Delta_S$-eigenfunctions extend to a universal $S$-invariant domain $\Xi_S\subset S_\C$. In fact, $\Xi_S$ is closely related to the crown domain
$\Xi\subset Z_\C=G_\C/K_\C$ of $Z$, and we refer to Section~\ref{section crown} for details. In particular, there exists a maximal domain $0\in \Lambda\subset \nf = \rm{Lie}(N)$ such that
\begin{equation} \label{XiS}\Xi_S \supset S \exp(i\Lambda)\,. \end{equation}
The domain $\Lambda$ has its origin in the unipotent model of the crown domain   \cite[Sect.~8]{KO} and, except in the rank one cases, its geometry is not known. Proposition \ref{prop bounded} implies that $\Lambda$ is bounded for a class of classical groups, including $G=\GL(n,\R)$. It is an interesting open problem whether $\Lambda$ is bounded or convex in general.

\par Now let $\phi: S\to \C$ be an eigenfunction of $\Delta_S$. For each $a\in A$ we define the
partial function
$$\phi_a: N \to \C, \quad n\mapsto \phi(na)\, .$$
Because eigenfunctions extend to $\Xi_S$, we see from \eqref{XiS}
that $\phi_a$ extends to a holomorphic function on the tube domain
\begin{equation}\label{defta}
\Tc_a:= N\exp(i\Lambda_a)\subset N_\C
\end{equation}
with
\begin{equation}\label{deflambdaa}
\Lambda_a= \Ad(a)\Lambda\, .
\end{equation}
The general perspective of this paper is to view
an eigenfunction $\phi$ as a family of holomorphic functions $(\phi_a)_{a\in A}$
with $\phi_a$ belonging to $\Oc(\Tc_a)$,  the space of all holomorphic functions on $\Tc_a$.

\par We now explain the Poisson transform and how eigenfunctions of the algebra $\mathbb{D}(Z)$ can be characterized
by their boundary values on $\partial Z$. Fix a minimal parabolic subgroup
$P=MAN$ with $M=Z_K(A)$. If $\theta: G\to G$ denotes the Cartan involution
with fixed point group $K$, we consider $\oline N=\theta(N)$ and the parabolic subgroup $\oline P= M A \oline N$ opposite to $P$.  Because $N\oline P\subset G$
is open dense by the Bruhat decomposition, it proves convenient to identify $\partial Z$ with $G/\oline P$. In the sequel we view $N\subset \partial Z=G/\oline P$ as an open dense subset.
\par For each $\lambda\in \af_\C^*$ one defines  the Poisson transform (in the $N$-picture) as
$$
\Pc_\lambda: C_c^\infty(N) \to C^\infty(S)\ ,
$$
\begin{equation}
\label{Poisson0} \Pc_\lambda f(s)= \int_N f(x) {\bf a} (s^{-1} x)^{\lambda + \rho} \ dx\
\qquad (s\in S)\ ,
\end{equation}
where ${\bf a}: KA\oline N \to A$ is the middle projection with respect to the opposite Iwasawa decomposition, { $a^\lambda:= e^{\lambda(\log a)}$ for $a\in A$} and $\rho { :=\frac{1}{2}\sum_{\alpha\in \Sigma^+} (\dim \gf^\alpha)\cdot \alpha}\in \af^*$ is the Weyl half sum with respect to $P$.

In this article we restrict to parameters $\lambda$ with  $\re \lambda (\alpha^\vee)>0$ for all
positive co-roots $\alpha^\vee\in \af$, denoted  in the following as $\re\lambda>0$. This condition ensures that the integral defining the Harish-Chandra ${\bf c}$-function
$${\bf c}(\lambda):=\int_N {\bf a}(n)^{\lambda+\rho} \ dn$$
converges absolutely.
\par Recall the Harish-Chandra isomorphism between
$\mathbb{D}(Z)$ and the $W$-invariant polynomials on $\af_\C^*$, where $W$ is
the Weyl group of the pair $(\gf, \af)$. In particular, $\Spec \mathbb{D}(Z)=\af_\C^*/ W$, and for each $[\lambda]=W\cdot \lambda$ we denote by $\E_{[\lambda]}(S)$ the corresponding eigenspace on $S\simeq Z$. The image of the Poisson transform consists of eigenfunctions,  $\operatorname{im} \Pc_\lambda(C_c^\infty(N))\subset
\E_{[\lambda]}(S)$.  Because ${\bf a}(\cdot)^{\lambda+\rho}$ belongs to $L^1(N)$ for $\re\lambda>0$, $\Pc_\lambda$ extends from $C_c^\infty(N)$ to $L^2(N)$.

The goal of this article is to characterize $\Pc_\lambda(L^2(N))$. As a first step towards this goal, for $f\in L^2(N)$ and $\phi=\Pc_\lambda(f)$ { in Lemma \ref{lemmaeasybound} we} note the estimate
 $$\|\phi_a\|_{L^2(N)} \leq a^{\rho -\re \lambda}{\bf c}(\re \lambda) \|f\|_{L^2(N)}$$
for all $a\in A$.

The basic observation in this paper is that the kernel $n\mapsto {\bf a}(n)^{\lambda+\rho}$ underlying
the Poisson transform  \eqref{Poisson0} extends holomorphically to $\Tc^{-1}:=\exp(i\Lambda)N$ and remains
$N$-integrable along every fiber, i.e.~for any fixed $y\in \exp(i\Lambda)$ the kernel $n\mapsto {\bf a}(yn)^{\lambda+\rho}$
is integrable over $N$.
This allows us to formulate a condition for positive left $N$-invariant continuous weight functions ${\bf w}_\lambda$ on the tubes  $\Tc=N\exp(i\Lambda)$, namely (see also \eqref{request w})
\begin{equation} \label{request intro w}\int_{\exp(i\Lambda)} {\bf w}_\lambda(y) \|\delta_{\lambda, y}\|^2_{ L^1(N)} \ dy <\infty\, ,\end{equation}
{ where the function $\delta_{\lambda, y}$ is defined in \eqref{deltadef}.} In the sequel we assume that ${\bf w}_\lambda$ satisfies condition \eqref{request intro w} and define rescaled weight functions
$${\bf w}_{\lambda,a}: \Tc_a\to \R_{>0}, \ \ ny\mapsto {\bf w}_\lambda(\Ad(a^{-1})y)\qquad (y\in\exp(i\Lambda_a))$$
on the scaled tubes $\Tc_a$.
The upshot then is that $\phi_a\in \Oc(\Tc_a)$ lies in the weighted Bergman
space
$$\B(\Tc_a, {\bf w}_{\lambda,a}):=\{ \psi\in \Oc(\Tc_a)\mid \|\psi\|^2_{\B_{a, \lambda}}:=
\int_{\Tc_a} |\psi(z)|^2 {\bf w}_{\lambda,a}(z)  dz <\infty\}$$
where $dz$ is the Haar measure on $N_\C$ restricted to $\Tc_a$.

This motivates us the definition of the following Banach  subspace of $\E_{[\lambda]}({ S})\subset \Oc(\Xi_S)$

$$\B(\Xi_S, \lambda):=\{ \phi \in \E_{[\lambda]}({ S})\mid
\|\phi\|:=\sup_{a\in A} a^{\re\lambda -2\rho} \|\phi_a\|_{\B_{a,\lambda}}<\infty\}\, .$$
It will be consequence of Theorem \ref{maintheorem} below that  $\B(\Xi_S, \lambda)$ as a vector space does not depend on the particular choice of the positive left $N$-invariant weight function ${\bf w}_\lambda$ satisfying \eqref{request intro w}.
The main result of this article 
 now reads:

\begin{theorem}\label{maintheorem}Let $Z=G/K$ be a Riemannian symmetric space and $\lambda\in \af_\C^*$ be a
parameter such that $\re \lambda>0$.  Then
$$\Pc_\lambda: L^2(N) \to \B(\Xi_S, \lambda)$$
is an isomorphism of Banach spaces, i.e. there exist $c,C>0$ depending on ${\bf w}_\lambda$ such that
$$c \|\Pc_\lambda(f)\|\leq \|f\|_{L^2(N)} \leq C \|\Pc_\lambda(f)\|\qquad (f\in L^2(N))\, .$$
\end{theorem}
Let us mention that the surjectivity of $\Pc_\lambda$ relies  on the established Helgason conjecture (see \cite{K6,GKKS}) and the Bergman inequality.

We now recall that $\Pc_\lambda$ is inverted by the boundary value map, that is
$${1\over {\bf c}(\lambda)} \lim_{a\to \infty\atop a\in A^-} a^{\lambda-\rho} \Pc_\lambda f(na) = f(n)\qquad (n\in N)\, ,$$
where the limit is taken along a fixed ray in the interior of the negative Weyl chamber { $A^-$}.

Define the positive  constant
\begin{equation} \label{def w const} w(\lambda):=\left[\int_{\exp(i\Lambda)} {\bf w}_\lambda(y) \ dy\right]^{1\over 2}.
\end{equation}
We observe that this constant is indeed finite, see Subsection \ref{sub:norm}.
{ There} we obtain  a corresponding norm limit formula:
\begin{theorem}\label{norm limit intro}
For any $f\in L^2(N)$, $\phi=\Pc_\lambda(f)$ we have
 \begin{equation} \label{norm limit2} {1\over w(\lambda) |{\bf c}(\lambda)|} a^{\re \lambda - 2\rho} \|\phi_a\|_{\B_{a,\lambda}} \to \|f\|_{L^2(N)} \qquad (f\in L^2(N))\end{equation}
for $a\to \infty$ on a ray in $A^-$.
\end{theorem}

Let us emphasize that the weight functions ${\bf w}_\lambda$  are not unique
and it is natural to ask about the existence of optimal choices, i.e.~choices for which $\Pc_\lambda$ establishes an isometry
between $L^2(N)$ and $\B(\Xi_S, \lambda)$, in other words whether a norm-sup identity holds:

\begin{equation} \label{norm sup} \sup_{a\in A} {1\over w(\lambda) |{\bf c}(\lambda)|} a^{\re \lambda - 2\rho} \|\phi_a\|_{\B_{a,\lambda}} =\|f\|_{L^2(N)} \qquad (f\in L^2(N))\, .\end{equation}
The answer is quite interesting in the classical example of the real hyperbolic space
$$Z=\SO_e(n+1,1)/\SO(n+1)\simeq \R^n \times \R_{>0} = N\times A$$
where the study was initiated in \cite{RT} and is now completed in Section \ref{sect hyp}.  Here $N=\nf=\R^n$ is abelian and we recall the classical formulas for the
Poisson kernel and ${\bf c}$-function
$${\bf a}(x)^{\lambda+\rho} = ( 1 +{ |x|}^2)^{-(\lambda +n/2)}\qquad (x\in N=\R^n)\, ,$$
$${\bf c}(\lambda)=  \pi^{n/2} \frac{\Gamma(2\lambda)}{\Gamma(\lambda+n/2)}\, , $$
{ where we write $|\cdot|$ for the Euclidean norm}. It is now easily seen that $\Lambda=\{ y \in \R^n \mid { |y|}<1\}$ is the open unit ball.
A natural family of weights to consider  are powers of the Poisson kernel parametrized by $\alpha>0$

\begin{equation} \label{special weight 1} {\bf w}_{\lambda}^\alpha(z) =
(2\pi)^{-n/2} \frac{1}{\Gamma(\alpha)} \left(1-{ |y|}^2\right)_+^{\alpha -1} \qquad (z=x+iy\in \Tc = \R^n +i\Lambda)\,  ,\end{equation}
{ where $(\,\cdot\, )_+$ denotes the positive part.} These weights satisfy condition \eqref{request intro w} exactly for
 $ \alpha > \max\{2s-1,0\} $ where $s=\re \lambda$, see Lemma \ref{deltabound}.
 Moreover, in Theorem \ref{thm hyp} we establish the following:
 \begin{enumerate}
 \item \label{one} Condition \eqref{request intro w} is only sufficient and Theorems
 \ref{maintheorem}, \ref{norm limit intro} hold even for $\alpha>\max\{2s-\frac{n+1}{2}, 0\}$.
 \item For $\alpha$ as in \eqref{one} and $\lambda=s>0$ real the norm-sup identity \eqref{norm sup} holds.
 \end{enumerate}
Let us stress that \eqref{norm sup}  is a new feature and isn't recorded (and perhaps not even true) for the range investigations with respect to the $K$-geometry in the rank one case: there{ ,} one  verifies lim-sup identities which are even weaker than the norm-limit
formula in Theorem \ref{norm limit intro}, see
 \cite{I}.

\section{Notation}
\label{sec:notation}

Most of the notation used in this paper is standard for semisimple Lie groups and symmetric spaces and can be found for instance in \cite{H3} { for the semisimple case and, for the general setting, in \cite{W1}.}

Let $G$ be the real points of a connected algebraic reductive group defined over $\R$ and let $\gf$ be its Lie algebra. Subgroups of $G$ are denoted by capitals. The corresponding subalgebras are denoted by the corresponding fraktur letter, i.e.~$\gf$ is the Lie algebra of $G$ etc.

\par  We denote by $\gf_\C=\gf\otimes_\R \C$ the complexification of $\gf$ and by $G_{\C}$ the group of complex points. We fix a Cartan involution $\theta$ and write $K$ for the maximal compact subgroup that is fixed by $\theta$.  We also write $\theta$ for the derived automorphism of $\gf$. We write $K_{\C}$ for the complexification of $K$, i.e.~$K_{\C}$ is the subgroup of $G_{\C}$ consisting of the fixed points for the analytic extension of $\theta$.

The Cartan involution induces the infinitesimal Cartan decomposition
$\gf =\kf \oplus\sf$. Let $\af\subset\sf$ be a maximal abelian
subspace.
The set of restricted roots of $\af$ in $\gf$ we denote by $\Sigma\subset \af^*\bs \{0\}$
and write $W$ for the Weyl group of $\Sigma$. We record the familiar root space decomposition
$$\gf=\af\oplus\mf\oplus \bigoplus_{\alpha\in\Sigma} \gf^\alpha\ ,$$
with $\mf=\zf_\kf(\af)$.

Let $A$ be the connected subgroup of $G$ with Lie algebra $\af$ and let $M=Z_{K}(\af)$.
We fix a choice of positive roots $\Sigma^+$ of $\af$ in $\gf$
and write $\nf=\bigoplus_{\alpha\in\Sigma^+} \gf^\alpha$ with corresponding unipotent subgroup
$N=\exp\nf\subset G$. As customary we set $\oline \nf =\theta(\nf)$ and accordingly
$\oline N = \theta(N)$.
For the Iwasawa decomposition $G=KA\oline N$ of $G$ we define the projections $\mathbf{k}:G\to K$ and $\mathbf{a}:G\to A$ by
$$
g\in \mathbf{k}(g)\mathbf{a}(g)\oline N\qquad(g\in G).
$$

Let $\kappa$ be the Killing form on $\gf$ and let $\tilde\kappa$ be a non-degenerate $\Ad(G)$-invariant symmetric bilinear form on $\gf$ such that its restriction to $[\gf,\gf]$ coincides with the restriction of $\kappa$ and $-\tilde\kappa(\,\cdot\,,\theta\,\cdot\,)$ is positive definite. We write $\|\cdot\|$ for the corresponding norm on $\gf$.

\section{The complex crown of a Riemannian symmetric space}\label{section crown}
The Riemannian symmetric space
$Z=G/K$ can be realized as a totally real subvariety of the Stein symmetric space
$Z_\C= G_\C/K_\C$:
$$ Z=G/K \hookrightarrow Z_\C, \ \ gK\mapsto gK_\C\, .$$
In the following we view $Z\subset Z_\C$ and write $z_0=K\in Z$ for the standard base point.

We define the subgroups $A_\C=\exp(\af_\C)$ and
$N_\C=\exp(\nf_\C)$ of $G_\C$. We denote by $F:=[A_\C]_{2-\rm{tor}}$ the finite group
of $2$-torsion elements and note that $F=A_\C \cap K$. Our concern is also with the solvable group $S=AN$ and its complexification
$S_\C=A_\C N_\C$. Note that $S\simeq Z$ as transitive $S$-manifolds, but
the natural morphism $S_\C\to Z_\C$ is neither onto nor injective. Its image
$S_\C \cdot z_0$ is Zariski open in the affine variety $Z_\C$ and we have
$S_\C/F \simeq S_\C\cdot z_0$.

The maximal $G\times K_\C$-invariant domain in
$G_\C$ containing $e$ and contained in $ N_\C A_\C K_\C$ is given by
\begin{equation} \label{crown1}
\tilde \Xi
= G\exp(i\Omega)K_\C\ ,
\end{equation}
where $\Omega=\{ Y\in \af\mid (\forall \alpha\in\Sigma)
\alpha(Y)<\pi/2\}$.
Note in particular that
\begin{equation} \label{c-intersect} \tilde \Xi=\left[\bigcap_{g\in G} g N_\C A_\C K_\C\right]_0\end{equation}
with $[\ldots ]_0$ denoting the connected component of $[\ldots]$ containing $e$.

Taking right cosets by $K_\C$, we obtain
the $G$-domain
\begin{equation}\label{crown2} \Xi:=\tilde \Xi/K_\C \subset Z_\C=G_\C/K_\C\ ,\end{equation}
commonly referred to as the {\it crown domain}. See
\cite{Gi} for the origin of the notion, \cite[Cor.~3.3]{KS} for the inclusion
$\tilde \Xi\subset  N_\C A_\C K_\C$  and \cite[Th.~4.3]{KO} for the maximality.

We recall that $\Xi$ is a contractible space. To be more precise, let $\hat\Omega=\Ad(K)\Omega$ and note that $\hat\Omega$ is an open convex subset of $\sf$. As a consequence of the Kostant convexity theorem it satisfies $\hat\Omega\cap\af=\Omega$ and $p_{\af}(\hat\Omega)=\Omega$, where $p_{\af}$ is the orthogonal projection $\sf\to\af$. The fiber map
$$
G\times_{K}\hat\Omega\to\Xi; \quad [g,X]\mapsto g\exp(iX)\cdot K_{\C}\ ,
$$
is a diffeomorphism by \cite[Prop.~4, 5 and 7]{AG}.  Since $G/K\simeq\sf$ and $\hat\Omega$ are both contractible, also $\Xi$ is contractible. In particular, $\Xi$ is simply connected.

\par As $\Xi\subset S_\C\cdot z_0$ we also obtain a realization
of $\Xi$ in $S_\C/F$ which, by the contractibility of $\Xi$ lifts to an $S$-equivariant
embedding of $\Xi\hookrightarrow S_\C$. We denote the image by $\Xi_S$.
Let us remark that $\Xi_S$ is not known explicitly in appropriate coordinates
except when $Z$ has real rank one, which was determined in \cite{CK}.

\par We recall ${\bf a}: G \to A$ the middle projection of the Iwasawa decomposition $G=KA\oline{N}$ and note that
${\bf a}$ extends holomorphically to
\begin{equation}\label{tilde Xi}
\tilde \Xi^{-1}
:=\{g^{-1}:g\in\tilde\Xi\}\ .
\end{equation}
Here we use that $\tilde \Xi\subset \oline N_\C A_\C K_\C$ as a consequence of
$\Xi\subset N_\C A_\C K_\C$ and the $G$-invariance of $\Xi$. Moreover, the simply connectedness of $\Xi$ plays a role  to achieve ${\bf a}: \tilde \Xi^{-1}\to A_\C$ uniquely: A priori ${\bf a}$ is only defined as a map to  $A_\C/F$.
We denote the extension of ${\bf a}$ to $\tilde \Xi^{-1}$ by the same symbol.

 Likewise one remarks that $\mathbf{k}: G \to K$ extends holomorphically to $\tilde \Xi^{-1}$ as well.

\subsection{Unipotent model for the crown}
Let us define a domain $\Lambda\subset \nf$ by

$$\Lambda:=\{ Y \in \nf\mid \exp(iY)\cdot z_0\subset \Xi\}_0$$
where the index $\{\cdot\}_0$ refers to the connected component of $\{\cdot\}$
containing $0$. Then we have
$$\Xi=G\exp(i\Lambda)\cdot z_0$$
by \cite[Th. 8.3]{KO}. In general the precise shape of $\Lambda$ is not known
except for a few special cases, in particular if the real rank of $G$ is one
(see \cite[Sect. 8.1 and 8.2]{KO}).

\begin{prop} \label{prop bounded} For $G=\GL(n,\R)$ the domain $\Lambda\subset \nf$ is bounded.
\end{prop}
\begin{rmk}\label{rmk bounded}  A general real reductive group $G$ can be embedded into $\GL(n,\R)$ with compatible Iwasawa
decompositions. Then it happens in a variety of cases that the crown domain $\Xi=\Xi(G)$ for $G$ embeds
into the one of $\GL(n,\R)$. For example this is the case for $G=\SL(n,\R), \Sp(n,\R), \Sp(p,q), \SU(p,q)$,   and we refer to \cite[Prop. 2.6]{KrSt} for a complete list. In all these cases  $\Lambda$ is then bounded as a consequence of Proposition
\ref{prop bounded}. \end{rmk}

\begin{proof}[Proof of Proposition \ref{prop bounded}] Define
$$\Lambda'=\{ Y \in \nf\mid  \exp(iY)N \subset K_\C A_\C \oline N_\C\}_0$$
and note that $\Lambda'=-\Lambda$. Now \eqref{c-intersect}  for $N$ replaced by $\oline N$ implies
$\Lambda\subset \Lambda'$. We will show an even stronger statement by replacing $\Lambda$ by $\Lambda'$; in other words we search for the largest tube domain
$T_{N,\Lambda'}:=\exp(i\Lambda') N$ contained in $K_\C A_\C \oline N_\C$ and show that the tube has bounded base.

As usual we let $K_\C=
\SO(n,\C)$, $ A_\C=\diag(n, \C^*)$ and $\oline N_\C$ be the unipotent
lower
triangular matrices. We recall the construction of the basic $K_\C\times \oline N_\C$-invariant functions
on $G_\C$. With $e_1, \ldots, e_n$ the standard basis of $\C^n$ we let $v_i:= e_{n-i+1}$, $1\leq i\leq n$.
Now for $1\leq k\leq n-1$ we define a holomorphic function on $G_\C = \GL(n,\C)$ by
$$f_k(g) = \det \left(\la g(v_i), g(v_j)\ra_{1\leq i,j\leq n-k}\right) \qquad (g\in G_\C)$$
where $\la z,w\ra = z^t w$ is the standard pairing of $\C^n$. As the standard pairing is $K_\C$-invariant we obtain
that $f_k$ is left $K_\C$-invariant. Furthermore from
$$f_k(g) =\la g(v_1)\wedge\ldots \wedge g(v_{n-k}), g(v_1)\wedge\ldots \wedge g(v_{n-k})\ra_{\bigwedge^{n-k}\C^n}$$
we see that $f_k$ is right-$\oline N_\C$-invariant. In particular we have
$$f_k(\kappa a\oline n)= a_{k+1} \cdot\ldots \cdot a_n \qquad (\kappa \in K_\C , \oline n\in \oline N_\C)$$
for $a=\diag(a_1, \ldots, a_n)\in A_\C$.
Hence $f_k$ is not vanishing on $K_\C A_\C \oline N_\C$ and in particular not on the tube domain
$T_{N,\Lambda'}$ which is contained in $K_\C A_\C \oline N_\C$.
\par The functions $f_k$ are right semi-invariant under the maximal parabolic subgroup
$\oline P_k = L_k \oline U_k$ with $L_k=\GL(k,\R)\times \GL(n-k,\R)$ embedded block-diagonally
and $\oline U_k =\1_n+ \Mat_{(n-k)\times k}(\R)$ with $\Mat_{(n-k)\times k }(\R)$ sitting in the lower left corner.
We obtain with $U_k= \1_n+ \Mat_{k\times (n-k)}(\R)$ an abelian subgroup of $N$ with $\uf_k = \Mat_{k \times (n-k)}(\R)$
and record for $Z=X+iY\in \uf_{k,\C}$ that
$$f_k(\exp(Z))= \det (\1_{n-k} + Z^t Z)\, .$$
From this we see that the largest $U_k$-invariant tube domain in $U_{k,\C}=\Mat_{k\times (n-k)}(\C)$ to which $f_k$ extends to a zero free holomorphic function is given by
$$T_k = \Mat_{k\times(n-k)}(\R) + i \Upsilon_k$$
where
$$\Upsilon_k=\{ Y\in \Mat_{k\times(n-k)}(\R)\mid  \1_{n-k}- Y^tY\ \hbox{is positive definite} \}$$
is bounded and convex.

\par With $\nf_k = \lf_k\cap \nf$ we obtain a subalgebra of $\nf$ such that $\nf = \nf_k \ltimes \uf_k$ is
a semi-direct product with abelian ideal $\uf_k$. Accordingly we have $N\simeq U_k \times N_k$ under the multiplication map and likewise we obtain,  via Lemma \ref{lemma bipolar} below,
for each $k$ a diffeomorphic polar map
$$\Phi_k: \uf_k \times \nf_k \times N \to N_\C, \ \ (Y_1, Y_2, n)\mapsto \exp(iY_1)\exp(iY_2)n\, .$$
Note that
$$\Phi_k^{-1}(T_{N,\Lambda'})=\Lambda_k'\times N$$
with $\Lambda_k'\subset \uf_k\times \nf_k$ a domain containing $0$. Now let $\Lambda_{k,1}'$ be the projection of $\Lambda_k'$ to $\uf_k$ and likewise we define $\Lambda_{k,2}'\subset \nf_k$.  Note that $\Lambda_k'\subset \Lambda_{k,1}'\times \Lambda_{k,2}'$. We now claim that $\Lambda_{k,1}'\subset \Upsilon_k$. In fact
let $Y=Y_1+Y_2 \in \Lambda_k'$. Then $\exp(iY_1)\exp(iY_2)\in T_{N,\Lambda'}$ and thus, as $f_k$ is right
$N_{k,\C}$-invariant,
$$ 0\neq f_k(\exp(iY_1)\exp(iY_2))=f_k(\exp(iY_1))\,.$$
Our claim follows.

\par To complete the proof we argue by contradiction and assume that $\Lambda'$ is unbounded. We will show that
this implies that $\Lambda_{k,1}'$ becomes unbounded, a contradiction to the claim above. Suppose now that
there is an unbounded sequence $(Y^m)_{m\in \N}\subset \Lambda'$. We write elements
$Y\in \nf$ in coordinates $Y=\sum_{1\leq i <j\leq n} Y_{i,j}$. Let now $1\leq k\leq n-1$
be maximal such that all $Y^{m}_{i,j}$ stay bounded for $j\leq k$. Our choice of parabolic subgroup then is
$\oline P_k$.

By assumption we have
that $Y^m_{i, k+1}$
becomes unbounded for some $1\leq i \leq k$. Let $l\geq i$ be maximal with this property.
We write elements $Y\in \nf$ as $Y_1+Y_2$ with $Y_1 \in \uf_k$ and
$Y_2\in \nf_k$.
Now for any $Y=Y_1+Y_2\in \nf$ we find unique $\tilde Y_1, X\in \uf_k$ such that
\begin{equation} \label{triple exp} \exp(iY)=\exp(i(Y_1 +Y_2))= \exp(i\tilde Y_1) \exp(iY_2)\exp(X)\end{equation}
as a consequence of the fact that $\Phi_k$ is diffeomorphic and the identity
$$\exp(iY) U_{k, \C} = \exp(iY_2) U_{k,\C}$$
in the Lie group $N_\C/ U_{k,\C}$.

By Dynkin's formula and the abelianess of $\uf_k$ we infer from \eqref{triple exp}
$$iY= ((i\tilde Y_1*iY_2)*X)=i\tilde Y_1 +iY_2+X+\sum_{j=1}^{n-1} c_j i^{j+1} (\ad Y_2)^j \tilde Y_1
+\sum_{j=1}^{n-1} d_j i^j (\ad Y_2)^j X$$
for certain rational constants $c_j, d_j\in \Q$. In particular, comparing real and imaginary parts on both sides we obtain two equations:
\begin{equation} \label{matrix1} Y_1 =  \tilde Y_1 +\sum_{j=1}^{n_1} c_{2j}(-1)^j
 (\ad Y_2)^{2j} \tilde Y_1 +\sum_{j=0}^{n_2} d_{2j+1} (-1)^{j} (\ad Y_2)^{2j+1} X \end{equation}
\begin{equation} \label{matrix2} X=   \sum_{j=0}^{n_1} c_{2j+1}(-1)^j
 (\ad Y_2)^{2j+1} \tilde Y_1 -\sum_{j=1}^{n_2} d_{2j} (-1)^{j} (\ad Y_2)^{2j} X, \end{equation}
{  where $n_1=\lfloor \frac{n-1}{2}\rfloor$ and $n_2=\lceil \frac{n-1}{2}-1\rceil$. }

Our claim now is that $(\tilde Y_1^m)_{l, k+1}$ is unbounded.  If $l=k$, then we deduce from
\eqref{matrix1} that $(Y_1^m)_{k, k+1}= (\tilde Y_1^m)_{k, k+1}$ is unbounded, i.e., our desired contradiction. Now suppose  $l<k$.
We are interested in the entries of $\tilde Y_1$ in the first column
and for that we let $\pi_1: \uf_{k,\C}=\Mat_{k\times (n-k)} (\C) \to \C^k$ { be} the projection { onto} the first column. We decompose $\lf_k=\lf_{k,1} +\lf_{k,2}$ with
$\lf_{k,1}= \gl(k, \R)$ and $\lf_2=\gl(n-k,\R)$.
Write $\uf_{k,j}=\R^k$ for the subalgebra of $\uf_k$ consisting of the $j$-th column and observe
\begin{align} \label{pi1}
\pi_1([\lf_{k,2}\cap \nf_k, \uf_k])&=\{0\}\\
\label{lfk1} [\lf_{k,1}, \uf_{k,j}]&\subset \uf_{k,j}.
\end{align}

Now write $Y_2 = Y_{2|1} + Y_{2|2}$ according to $\lf_{k}=\lf_{k,1}+\lf_{k,2}$.
From \eqref{matrix1}--\eqref{matrix2} together with \eqref{pi1}--\eqref{lfk1}  we then derive that
\begin{align} \label{matrix3}\pi_1(Y_1) &=  \pi_1(\tilde Y_1) +\sum_{j=1}^{n_1} c_{2j}(-1)^j
 (\ad Y_{2|1})^{ 2j} \pi_1(\tilde Y_1)\\
 \notag & \quad +\sum_{j=0}^{n_2} d_{2j+1} (-1)^{j} (\ad Y_{2|1})^{2j+1} \pi_1(X) \end{align}
and
\begin{equation} \label{matrix4} \pi_1(X)=   \sum_{j=0}^{n_1} c_{2j+1}(-1)^j
 (\ad Y_{2|1})^{ 2j} \pi_1(\tilde Y_1) -\sum_{j=1}^{n_2} d_{2j} (-1)^{j} (\ad Y_{2|1})^{2j} \pi_1(X) \, .\end{equation}

We apply this now to $Y=Y^m$ and note that $Y_{2|1}^m$ is bounded by the construction of $\oline P_k$.  From \eqref{matrix3} and \eqref{matrix4} we obtain that
$X^m_{k+1, k}=0$ and $(\tilde Y_1^m )_{k, k+1}= (Y_1^m)_{k, k+1}$ and recursively
we obtain that $X_{i, k+1}^m$ and $\tilde Y_{i, k+1}^m$ remain bounded for $i<l$. It then follows from \eqref{matrix3}, as $Y^m_{l, k+1}$ is unbounded, that $\tilde Y^m_{l, k+1}$ is unbounded. This is the desired contradiction and completes the proof of the proposition.
\end{proof}

\begin{lemma} \label{lemma bipolar}Let $\nf$ be a nilpotent Lie algebra, $N_\C$ a simply connected Lie group
with Lie algebra $\nf_\C$ and $N=\exp(\nf)\subset N_\C$. Let further $\nf_1, \nf_2\subset \nf$ { be} subalgebras with $\nf=\nf_1 +\nf_2$ (not necessarily direct). Suppose that
$\nf_1$ is abelian. Then the 2-polar map
$$\Phi:  \nf_1 \times\nf_2 \times N \to N_\C, \ \ (Y_1, Y_2, n) \mapsto \exp(iY_1) \exp(iY_2) n $$
is onto. If moreover, the sum $\nf_1+\nf_2$ is direct and $\nf_1$  is an ideal, then
$\Phi$ is diffeomorphic.
\end{lemma}
\begin{proof} We prove  the statement by induction on $\dim N$. Let $Z(N_\C)\subset N_\C$ the center of $N_\C$. Note that $Z(N_\C)$ is connected and of positive dimension if $\dim \nf>0$. Set $\tilde \nf:=\nf/\zf(\nf)$, $\tilde \nf_i:= (\nf_i +\zf(\nf))/\zf(\nf)$ and $\tilde N_\C =
N_\C/ Z(N_\C)$. Induction applies and we deduce that for every $n_\C \in N_\C$
we find elements $n\in N$, $Y_i\in \nf_i$ and $z_\C\in Z(N_\C)$ such that
$$ n_\C = \exp(iY_1) \exp(iY_2) n z_\C.
$$
We write $z_\C = z y $ with $z\in Z(N)$ and $y=\exp(iY)$ with $Y\in \zf(\nf)$. Write
$Y=Y_1' +Y_2'$ with $Y_i'\in \nf_i$. As $Y$ is central $Y_1'$ commutes with
$Y_2'$ and so $y =\exp(Y_1')\exp(Y_2')$. Putting matters together we arrive at
$$ n_\C = \exp(iY_1)\exp(iY_1') \exp(iY_2') \exp(iY_2) nz.
$$
Now $nz\in N$ and $\exp(iY_1)\exp(iY_1')=\exp(i(Y_1+Y_1'))$. Finally,
$\exp(iY_2')\exp(iY_2)= \exp(iY_2'')n_2$ for some $Y_2''\in \nf_2$ and
$n_2\in N_2 =\exp(\nf_2)$. This proves that $\Phi$ is surjective.
\par For the second part let us assume the further requirements. We confine ourselves with showing that $\Phi$ is injective.
So suppose that
$$\exp(iY_1)\exp(iY_2) n= \exp(iY_1') \exp(iY_2') n'$$
and reduce both sides mod the normal subgroup $N_{1,\C}$. Hence $Y_2=Y_2'$.
Since we have $N\simeq N_1\times N_2$ under multiplication we may assume, by the same argument that $n=n_1\in N_1$ and $n'=n_1'$. Now  injectivity  is immediate.
\end{proof}

\section{The Poisson transform and the Helgason conjecture}
\subsection{Representations of the spherical principal series}
Let $\oline P = M A\oline N$ and define for
$\lambda\in \af_\C^*$ the normalized character
$$\chi_\lambda: \oline P \to \C^*,\quad \oline p = ma\oline n \mapsto a^{\lambda-\rho}\,. $$
Associated to this character is the  line bundle $\Lc_\lambda= G\times_{\oline P} \C_\lambda\to G/\oline P$.
The sections of this line bundle form the representations of the spherical principal
series: We denote the $K$-finite sections by $V_\lambda$, the analytic sections
by $V_\lambda^\omega$ and  the smooth sections by $V_\lambda^\infty$.
Note in particular,
$$V_\lambda^\infty=\{ f\in C^\infty(G)\mid f(g\oline p ) =
\chi_\lambda(\oline p)^{-1}  f(g), \ \oline p\in \oline P, g \in G  \}$$
and that $V_\lambda^\infty$ is a $G$-module under the left regular representation.
Now given $f_1\in V_\lambda^\infty$ and $f_2\in V_{-\lambda}^\infty$ we obtain that
$f:=f_1f_2$ is a smooth section of $\Lc_{-\rho}$ which identifies with the 1-density bundle of the compact flag variety $G/\oline P$. Hence we obtain a natural $G$-invariant non-degenerate pairing

\begin{equation} \label{dual}V_\lambda^{\infty}\times V_{-\lambda}^\infty\to \C, \quad (f_1, f_2)\mapsto \la f_1, f_2\ra:=\int_{G/\oline P} f_1f_2\, .\end{equation}
In particular, the Harish-Chandra module dual to $V_\lambda$ is isomorphic to $V_{-\lambda}$.
The advantage using the pairing \eqref{dual}  is that it easily gives formulas when trivializing
$\Lc_\lambda$, and one securely obtains correct formulas for the compact and non-compact
picture. Using this pairing we define { the space of} distribution vectors as the strong dual
$V_\lambda^{-\infty}=(V_{-\lambda}^\infty)'$. Likewise we obtain { the space of} hyperfunction vectors
$V_\lambda^{-\omega}$. Altogether we have the natural chain
$$ V_\lambda\subset V_\lambda^\omega\subset V_\lambda^\infty\subset
V_\lambda^{-\infty} \subset V_{\lambda}^{-\omega}\, .$$
We denote by $f_{\lambda, K}\in V_\lambda$ the $K$-fixed vector with
$f_{\lambda, K}(\1)=1$ and normalize the identification of $\Lc_{-\rho}$ with the 1-density
bundle such that $\int f_{-\rho, K}=1$.

\subsection{Definition of the Poisson transform and Helgason's conjecture}
  We move on with the concept of Poisson transform and the Helgason conjecture on
$Z=G/K$ which was formulated in \cite{H1} and first established in \cite{K6}; see also \cite{GKKS} for a novel elementary  treatment.  We denote by ${\mathbb D}(Z)$ the commutative algebra of $G$-invariant differential
operators and recall that the Harish-Chandra homomorphism for $Z$ asserts that  ${\mathbb D}(Z)\simeq \Pol(\af_\C^*)^W$ with $W$ the Weyl group. In particular,
$\Spec{\mathbb D}(Z)\simeq \af_\C^*/W$.

For $\lambda\in \af_\C^*$
we denote by $\E_{[\lambda]}(Z)$ the ${\mathbb D}(Z)$-eigenspace attached to
$[\lambda]=W\cdot \lambda\in \af_\C^*/W$. Note that all functions in $\E_{[\lambda]}(Z)$ are eigenfunctions of $\Delta_Z$ to the eigenvalue $\lambda^2 - \rho^2$,
with $\lambda^2$ abbreviating the Cartan--Killing pairing $\kappa(\lambda, \lambda)$.
In case $Z$ has real rank one, let us remark that this characterizes
$\E_{[\lambda]}(Z)$, i.e.

$$ \E_{[\lambda]}(Z)=\{ f \in C^\infty(Z)\mid \Delta_Z f = (\lambda^2 -\rho^2)f\}\, . $$
For $\lambda\in \af_\C^*$ one defines the $G$-equivariant Poisson transform

$$\Pc_\lambda: V_\lambda^{-\omega}\to C^\infty(G/K),
\ \ f\mapsto (gK\mapsto  \la f, g\cdot f_{-\lambda, K}\ra).
$$
The Helgason conjecture then asserts that $\Pc_\lambda$ is onto the
$\mathbb{D}(Z)$-eigenspace  $\E_{[\lambda]}(Z)$
provided that $f_{K,-\lambda}$ is cyclic in $V_{-\lambda}$, i.e.
$\U(\gf)f_{K,-\lambda}= V_{-\lambda}$.   The latter condition is always satisfied
if Kostant's condition \cite[Th.~8]{Kos} holds: $\re \lambda(\alpha^\vee)\geq 0$ for all positive roots $\alpha$.
In the sequel we abbreviate this condition as $\re \lambda \geq 0$.

If $\re \lambda >0$, then the Poisson transform is inverted
by the boundary value map
$$b_\lambda: \E_\lambda(Z) \to V_\lambda^{-\omega}, \ \
\phi\mapsto (g\mapsto {\bf c}(\lambda)^{-1}\lim_{a\to \infty\atop a\in A^-}  a^{\lambda -\rho} \phi(ga))$$
where ${\bf c}(\lambda)$ is the Harish-Chandra ${\bf c}$-function:

$${\bf c}(\lambda):=\int_N{\bf a}(n)^{\lambda +\rho} \ dn $$
with ${\bf a}: KA\oline N \to A$ the middle projection.

 In particular, we have \begin{equation} \label{boundary} b_\lambda(\Pc_\lambda(f)) = f \qquad (f \in V_\lambda^{-\omega}, \re \lambda >0)\, .\end{equation}

\section{The Poisson transform in terms of $S$-geometry}
{ As emphasized in the introduction our focus in this article is on the $S=AN$-picture of $Z=G/K$ which we henceforth identify with $S$. In particular, we will write
$\E_{[\lambda]}(S)$ instead of $\E_{[\lambda]}(Z)$ etc.}

\par
We fix a parameter
$\lambda$ such that $\re \lambda >0$. The goal is to identify
subspaces of $V_\lambda^{-\omega}$ for which $\Pc_\lambda$ has a particularly
nice image in terms of $S$-models. From what we already explained we
have
$$ \operatorname{im} \Pc_\lambda\subset \Oc(\Xi_S)$$
and, in particular, for  all $\phi  \in \operatorname{im} \Pc_\lambda$ and $a\in A$
we have $\phi_a\in  \Oc(\Tc_a)$.
The general problem here is that one wants to identify
$V_\lambda^{-\omega}$ with a certain subspace of $C^{-\omega}(N)$ which is tricky and
depends on the parameter $\lambda$.  The compact models for the spherical principal series here are much cleaner to handle as the restriction maps
$$\res_{K,\lambda} : V_\lambda^\infty \to C^\infty(K/M)=C^\infty(K)^M, \quad f\mapsto f|_K$$
are isomorphisms. In this sense we obtain a natural identification
$V_\lambda^{-\omega} \simeq C^{-\omega}(K/M)$ as $K$-modules which is
parameter independent.
Contrary to that, the faithful restriction map
$$\res_{N,\lambda} : V_\lambda^\infty \to C^\infty(N), \quad f\mapsto f|_N$$
is not onto and the image depends on $\lambda$.
For a function $h\in C^\infty(N)$ we define a function { $H_\lambda$} on
the open Bruhat cell $NMA\oline N$ by
$$H_\lambda(n ma\oline n) =  h(n) a^{-\lambda+\rho}\, .$$
Then the image of $\res_{N,\lambda}$ is by definition given by
$$ C_\lambda^\infty(N)=\{ h \in C^\infty(N)\mid H_\lambda\ \hbox{extends to a smooth function on $G$}\}\, .$$
In this sense $V_\lambda^{-\omega}$ corresponds in the non-compact model to
$$C_\lambda^{-\omega}(N)= \{ h \in C^{-\omega}(N)\mid H_\lambda|_{K\cap N \oline P}\ \hbox{extends to a hyperfunction on $K$}\}\, .$$

Having said this we take an element $f\in C_\lambda^{-\omega}(N)$ and observe that
the Poisson transform in terms of $S$ is  given by
\begin{equation} \label{Poisson} \Pc_\lambda f(s)= \int_N f(x) {\bf a} (s^{-1} x)^{\lambda + \rho} \ dx\
\qquad (s\in S)\end{equation}
with ${\bf a}: KA\oline N \to A$ the middle projection.  In accordance with
\eqref{boundary} we then have
$${1\over {\bf c}(\lambda)} \lim_{a\to \infty\atop a\in A^-} a^{\lambda-\rho} \Pc_\lambda f(na) = f(n)\qquad (n\in N)\,.$$

Let us note that
the Hilbert model of $\Hc_\lambda=L^2(K/M)\subset C^{-\omega}(K/M)=V_\lambda^{-\omega}$ of $V_\lambda$ corresponds in the non-compact picture to $L^2(N, {\bf a}(n)^{2 \re \lambda} dn)\supset L^2(N)$ and hence
$$L^2(N)\subset C^{-\omega}_\lambda(N)\qquad (\re \lambda\geq 0)\, .$$

\par 
The main objective now is to give a novel characterization of $\Pc_\lambda(L^2(N))$ for  $\re \lambda>0$.
{  For a function  $\phi$ on $S=NA$ and $a\in A$ we recall the partial functions on $N$ defined by
$$\phi_a(n)= \phi(na)\qquad (n\in N)\, .$$}
Now, given $f\in L^2(N)$ we let
$\phi:=\Pc_\lambda(f)$ and rewrite \eqref{Poisson} as
\begin{equation} \label{P rewrite} {1\over {\bf c}(\lambda)}  a^{\lambda-\rho}  \phi_a(n) = \int_N f(x)\delta_{\lambda, a}(n^{-1}x) \ dx \end{equation}
with
\begin{equation} \label{deltadefa} \delta_{\lambda, a}(x):= {1\over {\bf c}(\lambda)} a^{-2\rho} {\bf a} (a^{-1} x a)^{\lambda+\rho}
\qquad (x \in N)\, .\end{equation}
We first note that the condition $\re \lambda>0$ then implies that $\delta_{\lambda, a}$ is a Dirac-sequence
on $N$ for $a\to \infty$ on a ray in  the negative Weyl chamber.

\begin{lemma}\label{lemmaeasybound} Let  $\phi=\Pc_\lambda(f)$ for  $f\in L^2(N)$. Then the following assertions
hold:
\begin{enumerate}
\item $\phi_a\in L^2(N)$ for all $a\in A$.
\item $\|\phi_a\|_{L^2(N)} \leq a^{\rho -\re \lambda}{\bf c}(\re \lambda) \|f\|_{L^2(N)}$.
\end{enumerate}
\end{lemma}
\begin{proof} Both assertions are immediate from the fact that $\|\delta_{{ \lambda,a}}\|_{L^1(N)} \leq
\frac {{\bf c}(\re \lambda)}{|{\bf c}(\lambda)|}$, \eqref{P rewrite} { and Young's convolution inequality}.
\end{proof}

\subsection{Partial holomorphic extensions of eigenfunctions}
{ Recall { $\Tc_a$ and $\Lambda_a$ from \eqref{defta}, resp.~\eqref{deflambdaa}, and} that the Poisson transform $\phi = \Pc_\lambda(f)$ belongs to $\Oc(\Xi_S)$ with all partial functions $\phi_a$ extending to holomorphic functions on $\Tc_a$.
For $y\in\exp(i\Lambda_a)$ we thus can define
$$\phi_{a,y}(n):=\phi_a(n y)\qquad (n \in N)\, .$$
Let $\delta_\lambda:=\delta_{\lambda, 1}$ and put
\begin{equation}\label{deltadef}\delta_{\lambda, y}:  N \to \C , \quad x \mapsto \delta_\lambda(y^{-1} x )\ .\end{equation}}
\begin{lemma} The following assertions hold:
\begin{enumerate} \item The function ${\bf v}_\lambda(y):=\sup_{k\in K} |{\bf a}(y^{-1} k)^{\lambda +\rho}|$ is finite for all $y \in \exp(i\Lambda)$.\\
\item The function $\delta_{\lambda, y}$ is integrable with $L^1(N)$-norm
\begin{equation} \label{delta bound2} v_\lambda(y):=\|\delta_{\lambda, y}\|_{ L^1(N)}\leq {\bf v}_\lambda(y)\frac{{\bf c}(\re \lambda)}{|\bf c(\lambda)|}\, .\end{equation}
\end{enumerate}
\end{lemma}
\begin{proof}
Part (1) is a consequence of the fact that ${\bf a}: G\to A$, considered as a
map from $K\bs G \to A$,   extends holomorphically
to $\Xi^{-1}\to A_\C$ with $\Xi^{-1}$ considered as a subset of $K_\C \bs G_\C$, see
\eqref{tilde Xi}. \\
For the proof of (2) we note the identity \begin{equation} \label{delta bound} \delta_\lambda(y^{-1} x )=\delta_\lambda(x)  {\bf a}(y^{-1} {\bf k}(x))^{\lambda +\rho}
\qquad (x \in  N, y\in \exp(i\Lambda)),\end{equation}
where ${\bf k}: G \to K$ is defined by the opposite Iwasawa decomposition $G=KA\oline N$.
Combined with part (1), \eqref{delta bound} implies that for all $y \in \exp(i\Lambda)$ the function
$\delta_{\lambda, y}$
is integrable on $N$, with the asserted estimate \eqref{delta bound2} for its norm.
\end{proof}
{ For $g, x\in G_\C$ we use the standard abbreviation $x^g:=gxg^{-1}$.}

\begin{lemma}\label{lemma a bound} For $\re \lambda>0$, $f\in L^2(N)$ and $\phi=\Pc_\lambda(f)$ we have

\begin{equation}\label{upper a-bound}  \|\phi_{a, y}\|_{L^2(N)} \leq {|\bf c}(\lambda)| a^{\rho -\re \lambda}\|\delta_{\lambda, y^{a^{-1}}}\|_{L^1(N)}
\|f\|_{L^2(N)}\qquad (y\in \exp(i\Lambda_a))\, .\end{equation}
\end{lemma}
\begin{proof} From \eqref{P rewrite} we obtain
$$ {1\over {\bf c}(\lambda)}  a^{\lambda-\rho} \phi_{a,y}(n) = \int_N f(x)\delta_{\lambda, a}(y^{-1}n^{-1}x) \ dx $$
and thus
\begin{equation} \label{deltaest}  {1\over |{\bf c}(\lambda)|}  a^{\re \lambda-\rho}\| \phi_{a,y}\|_{L^2(N)} \leq
\|\delta_{\lambda, a}(y^{-1}\cdot)\|_{L^1(N)} \|f\|_{L^2(N)}.
\end{equation}
{ Next we unwind the defintions \eqref{deltadefa} and \eqref{deltadef} and apply the change of variable $x\mapsto a^{-1}xa$ on $N$:
\begin{align*} \|\delta_{\lambda, a}(y^{-1}\cdot)\|_{L^1(N)}&={a^{-2\rho} \over |{\bf c}(\lambda)|}\int_N  \left|{\bf a} (a^{-1} y^{-1} x a)^{\lambda+\rho}\right|\ dx \\
&={a^{-2\rho} \over |{\bf c}(\lambda)|} \int_N  \left|{\bf a} ((a^{-1} y^{-1}a) a^{-1} x a)^{\lambda+\rho}\right|\ dx \\
&={1\over |{\bf c}(\lambda)|} \int_N   \left| {\bf a} ((y^{-1})^{a^{-1}} x )^{\lambda+\rho}\right|\ dx =\|\delta_{\lambda, y^{a^{-1}}}\|_{L^1(N)} \, .\end{align*}
The assertion \eqref{upper a-bound} now follows from \eqref{deltaest}.}  \end{proof}

\subsection{A class of weight functions}\label{subsection weight functions}
We  now let ${\bf w}_\lambda: \exp(i\Lambda)\to \R_{>0}$ be any positive continuous function such that
\begin{equation} \label{request w} d(\lambda):=\int_{\exp(i\Lambda)} {\bf w}_\lambda(y) \|\delta_{\lambda, y}\|^2_{ L^1(N)} \ dy <\infty\end{equation}
and define a left $N$-invariant function on the tube $\Tc_a$ by

$${\bf w}_{\lambda,a}: \Tc_a\to \R_{>0}, \quad n y\mapsto {\bf w}_\lambda (\Ad(a^{-1})y)\qquad (y\in \exp(i\Lambda_a))\, .$$

\begin{rmk} In general we expect that $\Lambda$ is bounded. In view
of \eqref{delta bound2} one may then take
$${\bf w}_\lambda \equiv 1\ ,$$
as ${\bf v}_\lambda^{-2}$ is bounded from below by a positive constant.
Optimal choices for ${\bf w}_\lambda$ in special cases will be presented at the end of the article.
\end{rmk}

We  now show that $\phi_a\in \Oc(\Tc_a)$ belongs to the weighted Bergman
space
$$\B(\Tc_a, {\bf w}_{\lambda,a}):=\{ \psi\in \Oc(\Tc_a)\mid \|\psi\|^2_{\B_{a, \lambda}}:=
\int_{\Tc_a} |\psi(z)|^2 {\bf w}_{\lambda,a}(z)  dz <\infty\},$$
where $dz$ is the Haar measure on $N_\C$ restricted to $\Tc_a$.
In more precision with $d(\lambda)$ from \eqref{request w}
we record the following { lemma}
\begin{lemma} \label{lemma5.5} Let $\re \lambda>0$, $f\in L^2(N)$ and $\phi=\Pc_\lambda(f)$. Then we have the following inequality
\begin{equation}\label{normb1}
\|\phi_a\|_{\B_{a,\lambda}}
\leq  |{\bf c}( \lambda)| \sqrt{d(\lambda)}  a^{2\rho-\re \lambda}\|f\|_{L^2(N)}\, .\end{equation}
\end{lemma}
\begin{proof} Starting with \eqref{upper a-bound} the assertion follows from the estimate
\begin{align*} \notag \|\phi_a\|_{\B_{a,\lambda}}&\leq a^{\rho- \re \lambda} |{\bf c}(\lambda)| \left[\int_{\exp(i\Lambda^a)}
{\bf w}_\lambda(y^{a^{-1}}) \|\delta_{\lambda, y^{a^{-1}}}\|^2_{ L^1(N)} \ dy]\right]^{1\over 2}
\|f\|_{L^2(N)}\\
\notag&= |{\bf c}(\lambda)| \left[\int_{\exp(i\Lambda)}
{\bf w}_\lambda(y) \|\delta_{\lambda, y}\|^2_{ L^1(N)} \ dy\right]^{1\over 2}
 a^{2\rho-\re \lambda}\|f\|_{L^2(N)}\\
&=
 |{\bf c}(\lambda)| \sqrt{d(\lambda)}  a^{2\rho-\re \lambda}\|f\|_{L^2(N)}\,,
\end{align*}
as desired.
\end{proof}

{ Lemma \ref{lemma5.5}} motivates the definition of the following Banach  subspace of $\E_{[\lambda]}({ S})\subset \Oc(\Xi_S)$:

$$\B(\Xi_S, \lambda):=\{ \phi \in \E_{[\lambda]}({ S})\mid
\|\phi\|:=\sup_{a\in A} a^{\re\lambda -2\rho} \|\phi_a\|_{\B_{a,\lambda}}<\infty\}\, .$$
Indeed, \eqref{normb1} implies
\begin{equation} \label{P cont} \|\Pc_\lambda(f)\|\leq C \|f\|_{L^2(N)} \qquad (f\in L^2(N))\end{equation}
with  $C:={\bf c}(\re \lambda) \sqrt{d(\lambda)}$ and therefore the first inequality
in { Theorem \ref{maintheorem}}.

\begin{proof}[{ Proof of Theorem \ref{maintheorem}}] Since $\re\lambda>0$, the Poisson transform is injective. Further, \eqref{P cont}
shows that $\Pc_\lambda$ takes values in  $\B(\Xi_S, \lambda)$ and is continuous.
In view of the open mapping theorem,  it thus suffices to show that
$\Pc_\lambda$ is surjective. Note now that the weight ${\bf w}_\lambda$ is uniformly bounded above and below by positive constants when restricted to a compact subset $\exp(i\Lambda_c)\subset \exp(i\Lambda)$.  Hence the Bergman inequality implies the bound
\begin{equation} \label{norm 1} \|\psi|_N\|_{L^2(N)} \leq C a^{-\rho} \|\psi\|_{\B_{a,\lambda}}\quad (\psi \in \B(\Tc_a, {\bf w}_{\lambda, a})).\end{equation}
We apply this to $\psi=\phi_a$ for some $\phi\in \B(\Xi_S,\lambda)$ and obtain that
$a^{\lambda -\rho} \phi_a|_N $ is  bounded in $L^2(N)$. Hence we obtain for some sequence $(a_n)_{n\in \N}$ on a ray in $A^-$ that
$a_n^{\lambda-\rho} \phi_{a_n}|_N \to h$ weakly for some $h \in L^2(N)$.

By the Helgason conjecture we know that $\phi = \Pc_\lambda(f)$ for some $f\in C^{-\omega}_\lambda(N)$ and that

\begin{equation} \label{limit} {\bf c}(\lambda)^{-1} a^{\lambda -\rho} \phi_a|_N \to f\end{equation}
as appropriate hyperfunctions on $N$ for $a\to \infty$ in $A^-$ on a ray. Hence $h=f$ and we obtain
the second inequality of the theorem.
\end{proof}

\subsection{The norm limit formula}
\label{sub:norm}

Define a positive  constant
\begin{equation} \label{def w const} w(\lambda):=\left[\int_{\exp(i\Lambda)} {\bf w}_\lambda(y) \ dy\right]^{1\over 2} .\end{equation}
Note that $w(\lambda)$ is indeed finite. This will follow from \eqref{request w} provided we can show that $\|\delta_{\lambda, y}\|_1\geq 1$. Now, using Cauchy's theorem
we see that
\begin{equation} \label{cy} \int_N {\bf a} (y^{-1} n)^{\lambda +\rho} \ dn = {\bf c}(\lambda)\end{equation}
does not depend on $y\in \exp(i\Lambda)$. The estimate $\|\delta_{\lambda, y}\|_{ L^1(N)}\geq 1$ follows.

The purpose of this section is to prove the norm limit formula as stated in the introduction.
\begin{proof}[{ Proof of Theorem \ref{norm limit intro}}] In the sequel we first note that for any integrable function $\psi$ on $\Tc_a$ we have
$$
\int_{\Tc_a} |\psi(z)|^2\ dz = \int_{\Lambda_a} \int_N |\psi(yn)|^2 \ dn \ dY
$$
with $y=\exp(iY)$ and $dY$ the Lebesgue measure on $\nf$. With that
we rewrite the square of the left hand side of \eqref{norm limit2} as
\begin{align*} &{1\over w(\lambda)^2|{\bf c}(\lambda)|^2} a^{2\re \lambda - 4\rho} \|\phi_a\|_{\B_{a,\lambda}}^2= \\
&= {1\over w(\lambda)^2|{\bf c}(\lambda)|^2} a^{2\re \lambda - 4\rho} \int_{\Lambda_a}\int_N  |\phi_a(ny)|^2 {\bf w}_{\lambda, a} (y) \ dn \ dY  \\
&= {1\over w(\lambda)^2|{\bf c}(\lambda)|^2}a^{2\re \lambda - 2\rho} \int_{\Lambda}\int_N  |\phi_a(ny^a)|^2 {\bf w}_{\lambda, \1} (y) \ dn \ dY  \\
&= {1\over w(\lambda)^2|{\bf c}(\lambda)|^2}a^{2\re \lambda - 2\rho} \int_{\Lambda}\int_N \left|\int_N f(x) {\bf a} (y^{-1} a^{-1} n^{-1}  x)^{\lambda +\rho} \ dx \right|^2 {\bf w}_{\lambda, \1} (y) \ dn \ dY  \\
&={1\over w(\lambda)^2|{\bf c}(\lambda)|^2}  a^{-4\rho} \int_{\Lambda}\int_N \left|\int_N f(x) {\bf a} (y^{-1} a^{-1} n^{-1}  xa)^{\lambda +\rho} \ dx \right|^2 {\bf w}_{\lambda, \1} (y) \ dn \ dY\, .  \\
\end{align*}
Next we consider the function on $N$
$$\delta_{\lambda, y, a}(n):={1\over {\bf c}(\lambda)} a^{-2\rho} {\bf a} (y^{-1} a^{-1} na)^{\lambda +\rho} $$
and observe that this defines for any fixed $y\in \exp(i\Lambda)$ a Dirac-sequence when
$a\in A^-$ moves along a fixed ray to infinity, see \eqref{cy} for $\int_N \delta_{\lambda, a,y}= 1$.
We thus arrive at
\begin{multline*} {1\over w(\lambda)^2|{\bf c}(\lambda)|^2} a^{2\re \lambda - 4\rho} \|\phi_a\|_{\B_{a,\lambda}}^2\\
= {1\over w(\lambda)^2} \int_{\Lambda}\int_N \left|\int_N f(x) \delta_{\lambda, y, a} (n^{-1}x) \  dx \right|^2 {\bf w}_{\lambda, \1} (y) dn \ dY  \ .
\end{multline*}
We define a convolution type operator
$$T_{\lambda, y, a}: L^2(N) \to L^2(N), \quad f\mapsto \left(n \mapsto \int_N f(x) \delta_{\lambda, y, a}(n^{-1}x)\ dx\right) $$
and note that by Young's convolution inequality
$$\|T_{\lambda, y, a}(f)\|_{ L^2(N)} \leq \|\delta_{\lambda, y, a}\|_{ L^1(N)} \cdot \|f\|_{ L^2(N)}\ .$$
We continue with some standard estimates:
\begin{align*} &\left|\int_N \left|\int_N f(x) \delta_{\lambda, y, a} (n^{-1}x) \  dx \right|^2 \ dn \   - \|f\|^2_{ L^2(N)}\right|=  \left|\| T_{\lambda, y, a}(f)\|^2_{ L^2(N)} - \|f\|^2_{ L^2(N)}\right|\\
&\quad  = \left| \|T_{\lambda, y, a}\|_{ L^2(N)} - \|f\|_{ L^2(N)}\right| \cdot  (\|T_{\lambda, y, a}(f) \|_{ L^2(N)} + \|f\|_{ L^2(N)})\\
&\quad \leq \|T_{\lambda, y, a}(f) - f\|_{ L^2(N)} \cdot \|f\|_{ L^2(N)}( 1+ \|\delta_{\lambda, y, a}\|_{ L^1(N)})\\
&\quad =\left\|\int_N (f(\cdot x)  - f(\cdot)) \delta_{\lambda, y, a} (x) \  dx \right\|_{ L^2(N)} \cdot \|f\|_{ L^2(N)}( 1+ \|\delta_{\lambda, y, a}\|_{ L^1(N)})\\
&\quad \leq \|f\|_{ L^2(N)}( 1+ \|\delta_{\lambda, y, a}\|_{ L^1(N)})\int_N \|f(\cdot x) - f(\cdot)\|_{ L^2(N)} |\delta_{\lambda, y, a}(x)| \ dx \, .
\end{align*}
Now note that $x\mapsto \|f(\cdot x) - f(\cdot)\|_{ L^2(N)}$ is a bounded continuous function and $\frac{|\delta_{\lambda, y, a}|}{\|\delta_{\lambda, y,a}\|_{ L^1(N)}}$ is a Dirac-sequence
for $a\to \infty$ in $A^-$ on a ray. Hence we obtain a positive function
$\kappa_f(a)$ with $\kappa_f(a) \to 0$ for $a\to \infty$ in $A^-$ on a ray such that

$$\int_N \|f(\cdot x) - f(\cdot)\|_{ L^2(N)} |\delta_{\lambda, y, a}(x)| \ dx \leq \|\delta_{\lambda, y,a}\|_{ L^1(N)} \kappa_f (a)\, .$$

Putting matters together we have shown that
\begin{align*}& \left|{1\over |{\bf c}(\lambda)|^2} a^{2\re \lambda - 4\rho} \|\phi_a\|_{\B_{a,\lambda}}^2 -(\int_\Lambda {\bf w}_{\lambda,\1})\cdot \|f\|_{ L^2(N)}\right|\\
&\quad \le\kappa_f(a)\|f\|^2_{ L^2(N)} \int_{\Lambda} (1 +\|\delta_{\lambda, y, a}\|_{ L^1(N)}) \|\delta_{\lambda, y, a}\|_{ L^1(N)} {\bf w}_{\lambda, \1}(y) \ dy\ . \end{align*}
Finally observe that $\|\delta_{\lambda, y, a}\|_{ L^1(N)} =\|\delta_{\lambda, y}\|_{ L^1(N)}$
and hence
$$\int_{\Lambda} (1 +\|\delta_{\lambda, y, a}\|_{ L^1(N)}) \|\delta_{\lambda, y, a}\|_{ L^1(N)} {\bf w}_{\lambda, \1}(y) \ dy <\infty\ ,$$
by the defining condition \eqref{request w} for  ${\bf w}_\lambda$.
With that the proof of the norm limit formula \eqref{norm limit2}, i.e.~Theorem \ref{norm limit intro}, is complete.
\end{proof}

\section{The real hyperbolic space}\label{sect hyp}

In this section we investigate how the main results of this article take shape in the case of real hyperbolic spaces. After recalling the explicit formulas of the Poisson kernel we
provide essentially sharp estimates for $\|\delta_{\lambda, y}\|_{ L^1(N)}$ which allow us to perform the construction of a family of nice explicit weight functions ${\bf w}_\lambda$ satisfying \eqref{request w}.
These in turn have the property that for real parameters $\lambda=\re \lambda$ the weighted Bergman space $\B(\Xi_S, \lambda)$ becomes isometric to $L^2(N)$. In particular,
the Banach space $\B(\Xi_S, \lambda)$  is in fact a Hilbert space for the exhibited family of weights.

\subsection{Notation} Our concern is with the real hyperbolic space  $ \mathbf{H}_n(\R) = G/K $ where $ G = \SO_e(n+1,1)$ and  $K = \SO(n+1)$ for $n\geq 1$. Here $\SO_e(n+1,1)$ is the identity component of the group $\SO(n+1,1)$.    The Iwasawa decomposition $ G = KAN $ is given by $ N = \R^n$, $K = \SO(n+1) $ and $ A = \R_+.$  and we can identify $ \mathbf{H}_n(\R) $ with the upper half-space $ \R^{n+1}_+  = \R^n \times \R_+ $  equipped with the Riemannian metric $ g = a^{-2} (|dx|^2+da^2 ).$ 
For any $ \lambda \in \C $  which is not a pole of $\Gamma(\lambda+n/2)$  we consider the normalized kernels
 $$  p_\lambda(x, a) =   \pi^{-n/2} \frac{\Gamma(\lambda+n/2)}{\Gamma(2\lambda)} a^{\lambda+n/2}(a^2+|x|^2)^{-(\lambda+n/2)}, $$
 which play the role of the normalized Poisson kernel when $ \mathbf{H}_n(\R) $ is identified with the group $ S = NA$, $N =\R^n$, $A = \R_+.$  In fact, with ${\bf a}: G \to A$ the Iwasawa projection
with { respect} to $G=KA\oline N$ as in the main text we record for $x\in N=\R^n$ that
$${\bf a}(x)^{\lambda+\rho} = ( 1 +|x|^2)^{-(\lambda +n/2)}\, .$$
Further we have
$${\bf c}(\lambda)=  \pi^{n/2} \frac{\Gamma(2\lambda)}{\Gamma(\lambda+n/2)} $$
so that
$$  p_\lambda(x, a) = {1\over {\bf c}(\lambda)} {\bf a}(a^{-1} x)^{\lambda +\rho}.
$$
In the sequel we assume that $s:=\re \lambda>0$ and note that $\rho=n/2$.

 The classical Poisson transform (normalize \eqref{Poisson} by ${1\over {\bf c}(\lambda)}$) of a function $ f \in L^2(\R^n) $ is then given by
 \begin{align*} \mathcal{P}_\lambda f(x,a) &= f*p_{\lambda}(\cdot, a)\\
 &=\pi^{-n/2}  \frac{\Gamma(\lambda+n/2)}{\Gamma(2\lambda)} a^{-(\lambda+n/2)}  \int_{\R^n} f(u) (1+a^{-2} |x-u|^2)^{-\lambda-n/2} du\, \end{align*}
with $\ast$ the convolution on $N=\R^n$. It is easy to check that $  \mathcal{P}_\lambda f(x,a) $ is an eigenfunction of the Laplace--Beltrami operator $ \Delta $ with eigenvalue $\lambda^2- (n/2)^2.$

 From the explicit formula for the Poisson kernel it is clear that for each $ a \in A $ fixed, $ \mathcal{P}_\lambda f(x, a) $ has a holomorphic extension to the tube domain
 $$ \Tc_a:=  \{ x+iy \in \C^n \mid |y| < a \} = N\exp(i\Lambda_a)\subset N_\C=\C^n, $$
 where $ \Lambda_a = \{ y \in \R^n :  |y| < a \}.$ Writing $ \phi_a(x) = \mathcal{P}_\lambda f(x,a) $ as in (\ref{P rewrite}) we see that
 $$
 \delta_{\lambda, y}(x)=  {1\over {\bf c}(\lambda)} (1+(x+iy)^2)^{-(\lambda+n/2)}.
 $$
 A weight function $ {\bf w}_\lambda $ satisfying (\ref{request w}), namely
 $$ d(\lambda) = \int_{|y| <1} {\bf w}_\lambda(y) \|\delta_{\lambda,y}\|^2_{ L^1(\R^n)} \, dy < \infty$$
 can be easily found. Indeed, as
 $$ (1+z^2)^{-(n/2+\lambda)} = \frac{2^{-n-\lambda}}{\Gamma(\lambda+n/2)}  \int_0^\infty e^{-\frac{1}{4t} (1+z^2)} t^{-n/2-\lambda-1} dt $$
where $ z^2 = z_1^2+z_2^2+...+z_n^2$ we have
$$ |\delta_{\lambda,y}(x)| \leq c_\lambda \int_0^\infty e^{-\frac{1}{4t} (1-|y|^2+|x|^2)} t^{-n/2-s-1} dt $$
valid for $ |y| <1.$ From this it is immediate that we have the estimate
$$  \|\delta_{\lambda,y}\|_{ L^1(\R^n)} \leq c_\lambda (1-|y|^2)_+^{-s}\, .$$

However this bound is not optimal and we can do better with slightly more effort. This will be part of the next subsection.

\subsection{Bounding $\|\delta_{\lambda, y}\|_{ L^1(\R^n)}$ and special weights.}
\begin{lemma}\label{deltabound} For $s=\re \lambda>0$ we have  for a constant $C=C(\lambda, n)>0$ that
$$\|\delta_{\lambda, y}\|_{ L^1(\R^n)} \asymp
\begin{cases*} C & if $0<s<\frac{1}{2}$,\\
C |\log(1-|y|^2)_+| & if $s=\frac{1}{2}$,\\
C (1-|y|^2)_+^{-s+\frac{1}{2}} & if $s>\frac{1}{2}$,
\end{cases*} \qquad \qquad (|y|<1).$$
\end{lemma}
\begin{proof}
To begin with we have
\begin{align*}  \|\delta_{\lambda,y}\|_{ L^1(\R^n)}&\asymp \int_{\R^n} |1+(x+iy)^2|^{-(n/2+s)}\ dx \\
 &\asymp \int_{\R^n} (1-|y|^2 +|x|^2+ 2|\la x, y\ra|)^{-(n/2+s)}\ dx\ .
\end{align*}
With $\gamma=\sqrt{1-|y|^2}$ we find
\begin{align*}  \|\delta_{\lambda,y}\|_{ L^1(\R^n)}&\asymp  \int_{\R^n} (\gamma^2 +|x|^2+ 2|\la x, y\ra|)^{-(n/2+s)}\ dx\\
&=  \int_{\R^n} (\gamma^2 +\gamma^2|x|^2+ 2 \gamma |\la x, y\ra|)^{-(n/2+s)}\ \gamma^n dx\\
&=  \gamma^{-2s}\int_{\R^n} (1 +|x|^2+ 2 |\la x, \gamma^{-1}y \ra|)^{-(n/2+s)}\ dx\ .
\end{align*}
Set
$$I_n(s,\gamma):=\int_{\R^n} (1 +|x|^2+ 2 |\la x, \gamma^{-1}y \ra|)^{-(n/2+s)}\ dx\,. $$
Then it remains to show that
\begin{equation} \label{Ins}  I_n (s,\gamma) \asymp\begin{cases*} \gamma^{2s} & if $0<s<\frac{1}{2}$, \\
\gamma  |\log \gamma| & if $s=\frac{1}{2}$, \\
 \gamma & if $s>\frac{1}{2}$
\end{cases*} \, .\end{equation}

We first reduce the assertion to the case $n=1$ and assume $n\geq 2$.
By rotational symmetry we may assume that $y=y_1 e_1$ is a multiple of the first unit vector with $1/2<y_1 <1$. Further we write $x=(x_1,x')$ with $x'\in \R^{n-1}$. Introducing polar coordinates $r=|x'|$, we find
\begin{align*}
&I_n(s,\gamma) = \int_{\R^n} (1+ |x'|^2 +x_1^2+2 \gamma^{-1}|x_1|y_1 )^{-(n/2+s)}\ dx\\
& \asymp \int_0^\infty  \int_0^\infty r^{n-2} (1+ r^2 +x_1^2+2 \gamma^{-1}x_1y_1 )^{-(n/2+s)} dx_1 \ dr \, .
\end{align*}
With $a^2:=1 + x_1^2 +2x_1 y_1 \gamma^{-1}$  this rewrites as
$$I_n(s, \gamma)\asymp \int_0^\infty  \int_0^\infty r^{n-2} (r^2 +a^2)^{-(n/2+s)} \ dr \ dx_1 $$
and with the change of variable $r=at$ we arrive at a splitting of integrals

\begin{align*} I_n(s,\gamma) &\asymp  \int_0^\infty  \int_0^\infty t^{n-2} (1+t^2)^{-\frac{n}{2} -s} a^{- n - 2s} a^{n-2} a \ dt  \ dx_1 \\
&= \underbrace{\left(\int_0^\infty t^{n-2} (1+t^2)^{-\frac{n}{2} -s} \ dt \right)}_{:=J_n(s)} \cdot  \underbrace{\left (\int_0^\infty
( 1 + x_1^2 +2 \gamma^{-1} x _1 y_1)^{-s -\frac{1}{2}}  \ dx_1\right)}_{=I_1(s,\gamma)}\end{align*}
Now $J_n(s)$ remains finite as long as $n\geq 2$ and $s>0$. Thus we have reduced
the situation to the case of $n=1$ which we finally address.
\par  It is easy to check that  $ I_1(s,\gamma) \asymp \gamma^{2s}$ for $ 0 < s < 1/2 $ and  $ I_1(s,\gamma) \asymp \gamma $ for $ s >1/2$.  When $ s = 1/2 $ we can evaluate $ \gamma^{-1} I_1(1/2,\gamma) $ explicitly. Indeed, by a simple computation we see that $ \gamma^{-1} I_1(1/2,\gamma)$ is given by
$$ 2 \int_0^\infty \frac{1}{(x_1+y_1)^2- (y_1^2-\gamma^2)} dx_1 =  \frac{-1}{ \sqrt{y_1^2-\gamma^2}}
\log \frac{y_1-  \sqrt{y_1^2-\gamma^2}}{y_1 + \sqrt{y_1^2-\gamma^2}}. $$
This gives the claimed estimate.
\end{proof}

For $\alpha>0$ we now define special weight functions by
 \begin{equation} \label{special weight} {\bf w}_\lambda^\alpha(z) =
(2\pi)^{-n/2} \frac{1}{\Gamma(\alpha)} \left(1-|y|^2\right)_+^{\alpha -1} \,  \qquad (z=x+iy\in \Tc)\, .\end{equation}
As a consequence of Lemma \ref{deltabound} we obtain

\begin{cor} The weight ${\bf w}_\lambda^\alpha$
satisfies the integrability condition \eqref{request w} precisely for
$$\alpha>\max\{2s-1, 0\}\, .$$
\end{cor}

\begin{rmk} Observe that ${\bf w}_{\lambda}^\alpha(z)$ is a power of the Iwasawa projection ${\bf a} (y)$. It would be interesting to explore this further in higher rank, i.e.
whether one can find suitable weights which are of the form
$${\bf w}_\lambda(ny)=|{\bf a}(y)^\alpha|\qquad { (}ny \in \Tc { )}$$
for some $\alpha=\alpha(\lambda)\in \af^*$.

\end{rmk}

For later reference we also record the explicit expression
\begin{equation} {\bf w}_{\lambda,a}^\alpha(z) =
(2\pi)^{-n/2} \frac{1}{\Gamma(\alpha)} \left(1-\frac{|y|^2}{a^2}\right)_+^{\alpha -1}\end{equation}
for the rescaled weights.

In the next subsection we will show that the general integrability
condition for the weight function \eqref{request w} is sufficient, but not sharp. By a direct use of the Plancherel theorem for the Fourier transform on $ \R^n $ we will show that one can do better for $\mathbf{H}_n(\R)$.\\

\subsection{Isometric identities}
Let  $K_\lambda$ be the Macdonald Bessel function
and $I_{\alpha+n/2} $ be the Bessel function of first kind with $\alpha>0$.  For $s:=\re \lambda >0$, we define non-negative weight functions
\begin{equation}\label{weigh}
 w_\lambda^\alpha(\xi): =   |\xi|^{2s}  \left|K_{\lambda}( |\xi|)\right|^2 \frac{I_{\alpha+n/2-1}(2|\xi|)}{(2|\xi|)^{\alpha+n/2-1}}\qquad (\xi \in \R^n)\, .\end{equation}

\begin{theorem} \label{thm level isometry} Let $\alpha>0, \lambda\in \C$, and
$s=\re \lambda>0$. There exists an explicit constant $c_{n,\alpha,\lambda} >0$ such that for all $f \in L^2(\R^n)$ and $\phi_a=\Pc_\lambda f(\cdot, a)$  we have the identity
\begin{equation} \label{level isometry} \int_{\Tc_a} |\phi_a(z)|^2 {\bf w}_{\lambda,a}^\alpha(z)\, dz =c_{n,\alpha,\lambda} \, a^{-2s+2n} \int_{\R^n}  |\widehat{f}(\xi)|^2  \, w_{\lambda}^\alpha(a \xi)  \,  d\xi \qquad (a>0)\end{equation}
 where  $ {\bf w}_\lambda^\alpha$ is as in \eqref{special weight}.
 \end{theorem}

\begin{proof} Let us set
$$ \varphi_{\lambda,a}(x) =  \pi^{-n/2}  \frac{\Gamma(\lambda+n/2)}{\Gamma(2\lambda)} (a^2+|x|^2)^{-(\lambda+n/2)} $$ so that we can write $ \phi_a(z) = \mathcal{P}_\lambda f(z,a) = a^{\lambda+n/2} f \ast \varphi_{\lambda,a}(z).$ In view of the Plancherel theorem  for the Fourier transform we have
$$ \int_{\R^n} |\phi_a(x+iy)|^2 dx = a^{2s+n} \int_{\R^n} e^{-2 y \cdot \xi} |\widehat{f}(\xi)|^2  |\widehat{\varphi}_{\lambda,a}(\xi)|^2 \, d\xi\, . $$
Integrating both sides of the above against the weight function ${\bf w}_{\lambda,a}^\alpha(z)$ we obtain the identity
\begin{equation} \label{main id} \int_{\Tc_a} |\phi_a(z)|^2 {\bf w}_{\lambda,a}^\alpha(z)dz = a^{2 s+n} \int_{\R^n}  |\widehat{f}(\xi)|^2  \, v_a^\alpha(\xi)  \, |\widehat{\varphi}_{\lambda,a}(\xi)|^2 \, d\xi\end{equation}
where $ v_a^\alpha(\xi) $ is the function defined by
$$ v_a^\alpha(\xi) = (2\pi)^{-n/2} \, \frac{1}{\Gamma(\alpha)} \,  \int_{|y| < a} e^{-2 y \cdot \xi}\, \left(1-\frac{|y|^2}{a^2}\right)_+^{\alpha-1}\ dy.$$
Both functions $ v_a^\alpha(\xi) $ and $\widehat{\varphi}_{\lambda,a}(\xi)$ can be evaluated explicitly in terms of Bessel and Macdonald functions. We begin with
$v_a^\alpha$ and recall that the Fourier transform of $(1-|y|^2)^{\alpha-1}_+$ is  explicitly known in terms of $J$-Bessel functions { (see \cite[Ch.~II, \S2.5]{GS})}:
$$
(2\pi)^{-n/2} \int_{\R^n} (1-|y|^2)^{\alpha-1}_+ e^{-i y\cdot \xi} dy =  \Gamma(\alpha) 2^{\alpha-1} |\xi|^{-\alpha-n/2+1}J_{\alpha+n/2-1}(|\xi|).
$$
 As the $J$-Bessel functions analytically extend to the imaginary axis, it follows that
 \begin{equation}
\label{FTweight}
(2\pi)^{-n/2} \, a^{-n}\,  \int_{\R^n} \left( 1-\frac{|y|^2}{a^2} \right)_+^{\alpha-1} e^{-2y\cdot \xi} dy = \Gamma(\alpha) 2^{\alpha-1}  \, (2a |\xi|)^{-\alpha-n/2+1} I_{\alpha+n/2-1}(2 a |\xi|)
\end{equation}
where $ I_{\alpha+n/2-1}$ is the modified Bessel function of first kind. We arrive at
\begin{equation} \label{vsa}
 v_a^\alpha(\xi)=2^{\alpha-1}  a^n (2a |\xi|)^{-\alpha-n/2+1} I_{\alpha+n/2-1}(2 a |\xi|)\, .\end{equation}

\par Moving on to $\widehat{\varphi}_{\lambda,a}(\xi)$ we use the integral representation
$$ \varphi_{\lambda,a}(x) =  \frac{(4 \pi)^{-n/2} 2^{-2\lambda}}{\Gamma(2\lambda)} \int_0^\infty e^{-\frac{1}{4t}(a^2+|x|^2)} t^{-n/2-\lambda-1} \, dt $$
and calculate the Fourier transform as
$$ \widehat{\varphi}_{\lambda,a}(\xi) = \frac{(2 \pi)^{-n/2} 2^{-2\lambda}}{\Gamma(2\lambda)} \int_0^\infty e^{-\frac{1}{4t}a^2}  \, e^{-t|\xi|^2} \,t^{-\lambda-1} \, dt\, . $$
The Macdonald function of type $ \nu $ is given by the integral representation
$$  r^\nu K_\nu(r) = 2^{\nu-1}   \int_0^\infty  e^{-t-\frac{r^2}{4t}} t^{\nu-1} dt,$$
for any $ r >0.$ In terms of  this function we have
\begin{equation} \label{phiK} \widehat{\varphi}_{\lambda,a}(\xi) = \frac{(2 \pi)^{-n/2} 2^{1-\lambda}}{\Gamma(2\lambda)} a^{-2\lambda} (a|\xi|)^\lambda K_\lambda(a|\xi|)\, .\end{equation}
Using these explicit formulas we obtain from \eqref{main id} that
$$ \int_{\Tc_a} |\phi_a(z)|^2 {\bf w}_{\lambda,a}^\alpha(z)dz =c_{n,\alpha,\lambda}\, a^{-2s+2n} \int_{\R^n}  |\widehat{f}(\xi)|^2  \, w_{\lambda}^\alpha(a \xi)  \,  d\xi$$
for an explicit constant $c_{n,\alpha,\lambda} $ and
\begin{equation} \notag w_\lambda^\alpha(\xi)=   |\xi|^{2s}  \left|K_{\lambda}( |\xi|)\right|^2 \frac{I_{\alpha+n/2-1}(2|\xi|)}{(2|\xi|)^{\alpha+n/2-1}},
\end{equation}
by \eqref{vsa} and \eqref{phiK}.
\end{proof}

In the sequel we write $\B_\alpha(\Xi_S, \lambda)$ to indicate the dependence on $\alpha>0$.
Now we are ready to state the main theorem of this section:

\begin{theorem}\label{thm hyp} For $\alpha>\max\{2s-\frac{n+1}{2}, 0\}$ the Poisson transform establishes
an isomorphism of Banach spaces
$$\Pc_\lambda: L^2(N)\to \B_\alpha(\Xi_S, \lambda)\,.$$
If moreover $\lambda=s$ is real, then $\Pc_\lambda$ is an isometry up to positive scalar.
\end{theorem}
\begin{proof} The behaviour of $ w_s^\alpha(\xi) $ can be read out from the  well known asymptotic properties of the functions $ K_\nu $ and $ I_\alpha.$ Indeed we can show (see  Proposition \ref{prop monotone} below) that there exists $c_1, c_2>0$ such that
$$ c_1 \, (1+|\xi|)^{-(\alpha+ \frac{n+1}{2}-2s)} \leq w_\lambda^\alpha(\xi) \leq c_2 \, (1+|\xi|)^{-(\alpha+\frac{n+1}{2}-2s)} \qquad (\xi \in \R^n)\, .$$
When $ 2s > \frac{n+1}{2},$ we can choose $ \alpha = (2s - \frac{n+1}{2}) >0 $ in
Theorem \ref{thm level isometry} above. With this choice, note that $ w_\lambda^\alpha(\xi) \leq c_3 $ and consequently we have
$$ \int_{\Tc_a} |\phi_a(z)|^2 {\bf w}_{\lambda,a}^\alpha(z)dz  \leq c_{n,\alpha,\lambda}\, a^{-2s+2n} \int_{\R^n}  |\widehat{f}(x)|^2    \,  dx .$$
This inequality implies that $\Pc_\lambda$ is defined. The surjectivity follows as in the proof of Theorem \ref{maintheorem} and with that the first assertion is established.   The second part is a consequence of Theorem
\ref{thm level isometry} and the stated monotonicity in Proposition \ref{prop monotone}.
\end{proof}

\begin{proposition}\label{prop monotone}
Let $\lambda = s>0$ be real and $\alpha>\max\{2s-\frac{n+1}{2}, 0\}$. Then $w_s^\alpha(r)$, $r >0 $, is positive and monotonically  decreasing. Moreover, $w_s^\alpha(0) =2^{-\alpha-\frac{n}{2}-2s-1} \frac{\Gamma(s)^2}{\Gamma(\alpha+\frac{n}{2})} $ and $w_s^\alpha(r) \sim c_\alpha \,  r^{-(\alpha+\frac{n+1}{2}-2s)}$ for $ r \to \infty$.
\end{proposition}
\begin{proof} From the definition, we have
$$ w_s^\alpha(r)=   r^{2s}  K_s(r)^2 \frac{I_{\alpha+n/2-1}(2r)}{(2r)^{\alpha+n/2-1}}.$$
Evaluating $ r^sK_s(r)$ and $ \frac{I_{\alpha+n/2-1}(2r)}{(2r)^{\alpha+n/2-1}}$ at $ r =0 $ by making use of the limiting forms close to zero (see  \cite[10.30]{OlMax})  we obtain $ w_s^\alpha(0) = 2^{-\alpha-\frac{n}{2}-2s-1} \frac{\Gamma(s)^2}{\Gamma(\alpha+\frac{n}{2})}$ as claimed in the proposition. The well known asymptotic properties of $ K_s(r) $ and $ I_\beta(r) $, see \cite[10.40]{OlMax}, proves the other claim. It therefore remains to show that $ w_s^\alpha(r) $ is monotonically decreasing, which will follow once we show that the derivative of $ w_s^\alpha(r)$ is negative at any $ r >0.$

Making use of the well known relations
$$
\frac{d}{dr}\big( r^s K_s(r)\big) = -r^s K_{s-1}(r),\qquad \frac{d}{dr}\Big(\frac{I_\beta(r)}{r^\beta}\Big) = \frac{I_{\beta+1}(r)}{r^\beta}$$
a simple calculation shows that
$$  \frac{d}{dr}\big( w_s^\alpha(r)\big) = 2 r^{2s} K_s(r) (2r)^{-(\alpha+n/2-1)}  F(r) $$
where $ F(r) = \Big( K_s(r) I_{\alpha+n/2}(2r) - K_{s-1}(r) I_{\alpha+n/2-1}(2r) \Big).$ Thus we only need to check that $ F(r) < 0 $ or equivalently
$$ \frac{I_{\alpha+n/2}(2r)}{I_{\alpha+n/2-1}(2r)} < \frac{K_{s-1}(r)}{K_s(r)}.$$
In order to verify this we make use of the following inequality proved by J.~Segura \cite[Theorem 1.1]{JS}: for any $ \nu\ge 0$ and  $r > 0 $ one has
$$ \frac{I_{\nu+1/2}(r)}{I_{\nu-1/2}(r)} <  \frac{r}{\nu+\sqrt{\nu^2+r^2}} \leq \frac{K_{\nu-1/2}(r)}{K_{\nu+1/2}(r)}.$$
Replacing $ r $ by $ 2r $ in the first inequality, from the above we deduce that
$$ \frac{I_{\nu+1/2}(2r)}{I_{\nu-1/2}(2r)} <  \frac{r}{\nu/2+\sqrt{\nu^2/4+r^2}} \leq \frac{K_{(\nu-1)/2}(r)}{K_{(\nu+1)/2}(r)}.$$

In the above we choose  $ \nu= 2s -1 $ so that $ (\nu-1)/2 = s-1 $ and $ (\nu+1)/2 = s.$ With $ \beta = \alpha+(n-1)/2 $ we have
$$ \frac{I_{\beta+1/2}(2r)}{I_{\beta-1/2}(2r)}  <  \frac{r}{\beta/2+\sqrt{\beta^2/4+r^2}}  <  \frac{r}{\nu/2+\sqrt{\nu^2/4+r^2}}\leq \frac{K_{(\nu-1)/2}(r)}{K_{(\nu+1)/2}(r)}$$
provided $ \beta > \nu$, i.e. $ \alpha+(n-1)/2 > 2s-1$ which is precisely the condition
on  $\alpha$ in the statement of the proposition.
\end{proof}

\section{Remarks on the extension problem}
We recall that given a function $f \in L^2(N)$ the Poisson transform $\phi =\Pc_\lambda(f)$  for $\re \lambda>0$, viewed as a function on $S$, gives us a solution
$$\Delta_S \phi ={ ( \lambda^2 -\rho^2)} \phi$$
such that we can retrieve $f$ through the boundary value map
$$b_\lambda(f)(n) = {\bf c}(\lambda)^{-1}\lim_{a\to \infty\atop a\in A^-}  a^{\lambda -\rho} \phi(na)\, .$$
We normalize $\Pc_\lambda$ in the sequel by ${1\over {\bf c}(\lambda)} \Pc_\lambda$ and replace $\phi$ by $\psi(na) = a^{\lambda -\rho} \phi(na)$. Hence it is natural
to ask about the differential equation which $\psi$ satisfies. To begin with we first derive a formula for $\Delta_S$ in $N\times A$-coordiantes. Note that $\Delta_Z$ on $Z=G/K$ descends from the Casimir operator $\Cc$
on right $K$-invariant functions on $G$ and so we start with $\Cc$. We assume that $G$ is semisimple and let $\kappa:\gf\times \gf\to \R$ be the Cartan--Killing form.
For an orthonormal basis $H_1, \ldots, H_n$ of $\af$ with respect to $\kappa|_{\af \times \af} $ we form the operator $\Delta_A= \sum_{j=1}^n H_j^2$ viewed
as a left $G$-invariant differential operator on $G$. Likewise we define $\Delta_M$ with respect to an orthonornal basis of $\mf$ with respect to
$-\kappa|_{\mf \times \mf}$. Now for each root space $\gf^\alpha$, $\alpha\in \Sigma^+$  we choose a basis $E_\alpha^j$, $1\leq j \leq m_\alpha=\dim \gf^\alpha$, such that with
$F_\alpha^j:=-\theta(E_\alpha^j)\in \gf^{-\alpha}$ we have $\kappa(E_\alpha^j, F_\alpha^k)=\delta_{jk}$.
Having said all that we obtain

$$\Cc=\Delta_A -\Delta_M+ \sum_{\alpha>0}\sum_{j=1}^{m_\alpha} E_\alpha^j F_\alpha^j + F_\alpha^j E_\alpha^j\, .$$
Now $\kappa|_{\af \times \af}$ complex linearly extended identifies $\af_\C$ with $\af_\C^*$ and for $\lambda\in \af_\C^*$ we let $H_\lambda\in \af_\C$ such that
$\lambda = \kappa(\cdot, H_\lambda)$.  We further recall that $[E_\alpha^j, F_\alpha^j]=H_\alpha$.
Thus we can rewrite $\Cc$ as
$$\Cc=\Delta_A -H_{2\rho} + \Delta_M+ \sum_{\alpha>0}\sum_{j=1}^{m_\alpha} 2E_\alpha^j F_\alpha^j \, .$$
Now note that $2E_\alpha^j F_\alpha^j = 2E_\alpha^j E_\alpha^j+ 2E_\alpha^j( F_\alpha^j - E_\alpha^j)$ with
$F_\alpha^j -E_\alpha^j\in \kf$.  Therefore $\Cc$ descends on right $K$-invariant smooth functions on $G$ to the operator

$$\Delta_Z = \Delta_A -H_{2\rho} + \sum_{\alpha>0}\sum_{j=1}^{m_\alpha} 2E_\alpha^j E_\alpha^j \, .$$
Now if we identify $Z=G/K$ with $N\times A$ then we see that that a point $(n,a)$ the operator $\Delta_S$ is given in the separated form
$$\Delta_S = \Delta_A -H_{2\rho} + \underbrace{\sum_{\alpha>0}\sum_{j=1}^{m_\alpha} 2a^{2\alpha}E_\alpha^j E_\alpha^j}_{=: \Delta_N^a}$$
with $\Delta_N^a$ acting on the right of $N$.  Hence
$$\Delta_S = \Delta_A -H_{2\rho} +\Delta_N^a\, .$$
Let now $\phi = \phi(n,a)$ be an eigenfunction of $\Delta_S$ to parameter $\lambda\in \af_\C^*$, say
$$\Delta_S \phi ={ ( \lambda^2 - \rho^2)} \phi\, .$$
We think of $\phi=\Pc_\lambda(f)$ for some generalized function $f$ on $N$ so that we can retrieve $f$ from $\phi$ via the boundary value map
\eqref{boundary}.  As motivated above we replace $\phi$ by $\psi(n,a)= a^{\lambda -\rho} \phi(n,a)$ and see what differential
equation $\psi$ satisfies. Since $\phi = a^{\rho-\lambda} \psi$ we have
$$(\Delta_A -H_{2\rho})a^{\rho-\lambda} \psi = a^{\rho -\lambda} ({ \lambda^2 -\rho^2} + \Delta_A - H_{2\rho} + \underbrace{{ 2} H_{\rho -\lambda}}_{ =H_{2\rho} - H_{2\lambda}})\psi\, .$$
This means that $\psi$ satisfies the differential equation
\begin{equation} \label{ext1}(\Delta_A - H_{ 2 \lambda} + \Delta_N^a)\psi(n,a)=0\, .\end{equation}
If we assume now that $\re \lambda>0$ and $\phi=\Pc_\lambda f$, then we have
\begin{equation} \label{ext2}\lim_{a\to \infty\atop a\in A^-} \psi(n,a)=  f(n)\end{equation}
The pair of equations \eqref{ext1} and \eqref{ext2} we refer to is the extension problem with parameter $\lambda$ for the operator $\Delta_S$. Given $f$ it has a unique solution provided
$\Delta_S$ generates all of $\mathbb{D}(Z)$, that is, the real rank of $G$ is one.

It is instructive to see what it is classically for the real hyperbolic spaces with $N=\R^n$. Here we use the classical (i.e. normalized Poisson transform)
and the extension problem becomes, for $(x,t)\in \R^n \times \R_{>0}$ (we use the notation from the previous section),
\begin{equation} \label{ext1a}
\big(\partial_t^2 + \frac{1-\lambda/2}{t} \partial_t+ \Delta_{\R^n}\big)\psi(x,t)=0\qquad (x\in \R^n, t>0)\end{equation}
If we assume now that $\re \lambda>0$ and $\phi=\Pc_\lambda f$, then we have
\begin{equation} \label{ext2a}\lim_{t\to 0} \psi(x,t)= f(x)\, .\end{equation}
The general theory then tells us that there is a unique solution $\psi$ of the extension problem \eqref{ext1a} and \eqref{ext2a} considered by Caffarelli and Silvestre  in \cite{CS}.

\section*{Acknowledgement}
The third author (L.~R.) is supported by Ikerbasque, by the Basque Government through the BERC 2022-2025 program, and by Agencia Estatal de Investigaci\'on through BCAM Severo Ochoa excellence accreditation CEX2021-001142-S/MCIN/AEI/10.13039/501100011033, RYC2018-025477-I, CNS2023-143893, and PID2023-146646NB-I00 funded by MICIU/AEI/10.13039/501100011033 and by ESF+.

\end{document}